\documentclass{article}
\usepackage{amsthm,amssymb,url,amsmath}
\usepackage{graphicx}
\usepackage{subcaption}
\usepackage[T1]{fontenc}
\usepackage{tgtermes}

\usepackage{color}

\usepackage{makecell}

\usepackage{fullpage}

\newtheorem{theorem}{Theorem}[section]

\newtheorem*{claim*}{Claim}

\newtheorem{corollary}[theorem]{Corollary}

\newtheorem{definition}[theorem]{Definition}

\newtheorem{lemma}[theorem]{Lemma}

\newtheorem{remark}[theorem]{Remark}

\newtheorem{statement}[theorem]{Statement}

\newcommand{\starv}{S{\scriptscriptstyle N}}
\newcommand{\UcovRes}{U{\scriptscriptstyle R}}
\newcommand{\IsolRes}{I{\scriptscriptstyle R}}

\newcommand{\comment}[1]{}

\author{S. Bereg\thanks{Department of Computer Science, University of Texas at Dallas, USA.}, L.E. Caraballo\thanks{Department of Applied Mathematics II, University of Seville, Spain. email:lcaraballo@us.es}, J.M. D\'iaz-B\'a\~nez\thanks{Department of Applied Mathematics II, University of Seville, Spain. email:dbanez@us.es}, M.A. Lopez\thanks{Department of Computer Science, University of Denver, USA.}, }

\title{Resilience of a synchronized multi-agent system}

\begin{document}
\maketitle

\begin{abstract}
Fault tolerance is increasingly important for unmanned autonomous vehicles. For example, in a multi robot system the agents need the ability to effectively detect and tolerate internal failures in order to continue performing their tasks without the need for immediate human intervention. The system must react to unplanned events in order to optimize the task allocation between the robots. In a broad sense, the resilience of a system can be defined as the ability to maintain or recover a stable state when subject to disturbance and it is related to the concept of robustness in industrial systems.

In this paper, we study the resilience in a synchronized multi-robot system stated as follows:
Consider a team of $n$ (ground or aerial) robots each moving along predetermined periodic closed trajectories. Each of the agents needs to communicate information
about its operation to other agents, but the communication links have a limited range. Hence, when
two agents are within communication range, a communication link is established, and information
is exchanged. Thus, two neighbors are synchronized if they visit the communication link at the same time and a multi-robot system is called synchronized if each pair of neighbors is synchronized. If a set of robots left the system, then some trajectories has no robots. In these cases, when an alive robot detects no neighboring robot then it pass to this neighboring trajectory to assume the unattended task.
 In this framework,  a fault-tolerance measure is introduced: the resilience of the system is the largest number of robots that can fail while executing the global task. Two tasks are considered: (a) no all the vehicles loss communication and, (b) all trajectories are covered. Thus, the two measures studied are called the isolation-resilience and the uncovering-resilience. Interesting combinatorial properties of the resilience are showed that allow to know its value for some usual scenarios.

\end{abstract}

\section{Introduction}

There exist a number of significant links between robotics and mathematics \cite{sack1999handbook,halperin1997robotics}. Recently, there has been increased research interest in systems composed of multiple autonomous mobile robots exhibiting cooperative behavior and the coordination of multi-agent systems is a very trendy research area and provides interesting mathematical problems \cite{kawamura2015simple, kawamura2015fence, dumitrescu2014fence, kranakis2011boundary, chevaleyre2004theoretical, acevedo_jint14, caraballo2014block}. On other hand, many algorithmic and combinatorial problems are very interesting to design coordinated multi-robot systems \cite{cusick1973view, wills1967zwei, alejo2013velocity}.

Scalability, fault-tolerance and failure-recovery have always been a challenge for a distributed systems. In a recent paper, \cite{dbanez2015icra} pose the following scenario: there is a team of $n$ aerial robots (robots) which are periodically traveling along predetermined closed trajectories while performing an assigned task. Each of the agents needs to communicate information about its operation to other agents, but the communication interfaces have a limited range. Hence, when two agents are within communication range, a communication link is established, and information is exchanged using the established communication graph. If two neighboring agents can exchange information periodically, we say they are ``synchronized''. Given the path geometries, the so-called \emph{synchronization problem} is to schedule (if possible) the movement agents along the trajectories so that every pair of neighboring agents is synchronized.

\begin{figure}
\begin{subfigure}{.245\textwidth}
\centering
\includegraphics[page=5]{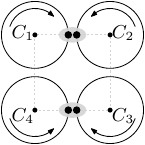}
\caption{}
\label{fig:synchro_sample_a}
\end{subfigure}
\begin{subfigure}{.245\textwidth}
\centering
\includegraphics[page=6]{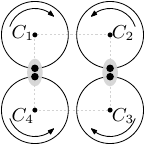}
\caption{}
\label{fig:synchro_sample_b}
\end{subfigure}
\begin{subfigure}{.245\textwidth}
\centering
\includegraphics[page=7]{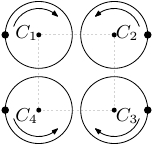}
\caption{}
\end{subfigure}
\begin{subfigure}{.245\textwidth}
\centering
\includegraphics[page=8]{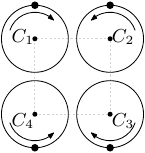}
\caption{}
\end{subfigure}
\caption{Synchronized system of aerial robots following circular trajectories. The robots are represented as solid points in the circles. 
In this example the communication graph is a cycle.}
\label{fig:synchro_sample}
\end{figure}

Figure~\ref{fig:synchro_sample} shows a synchronized system where every pair of neighboring robots are flying in opposite directions (one of them in clockwise direction and the other in counterclockwise) at the same constant speed. We have enclosed in a gray shadow the pairs of neighboring robots with an established communication link between them. Each subfigure represents a state of the system every quarter of a period (time to make a tour in a trajectory) starting at the state in Figure~\ref{fig:synchro_sample_a}.

This problem can appear in missions of surveillance or monitoring \cite{acevedo_jint14, pasqualetti2012cooperative}, in structures assembly while the robots are loading and placing parts in the structure \cite{??} but many other applications exist. In fact, its eventual solution may find many applications beyond this initial problem.

%
%
The authors solve the problem in a simplified model where all agents traverse
around unit circles at uniform speed (like in Figure~\ref{fig:synchro_sample}) and then more general models are also considered.
In the paper, the authors also considered an interesting variation of the
problem where dropouts in one or more agents are handled by the surviving agents.
In order to minimize
the detrimental effect of the departing robots on global system
performance, the proposed strategy consists of repeatedly
``switching'' the trajectory when a neighbor leaves the system. Suppose that a set of robots left the system, then some trajectories has no robots. In these cases, when an alive robot (a robot that has not left the system) arrives to a link position with a trajectory without robot and it detects no neighboring robot then it passes to this neighboring trajectory, see Figure~\ref{fig:ch_routes}.
%
\begin{figure}[h]
\centering
\includegraphics[scale=.8]{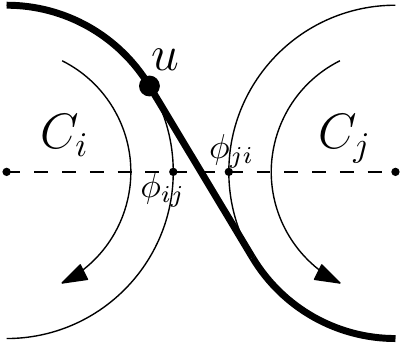}
\caption{A robot changes its trajectory when it detects that the corresponding neighbor has left the system. The robot $u$ follows the path drawn with fat stroke. The trajectory $C_j$ has no robot, then when the robot $u$ (in $C_i$) arrives to the link position pass to $C_j$.}
\label{fig:ch_routes}
\end{figure}

However, in some cases, when a set of robots leaves the team an undesirable phenomenon may occur: a robot, independent of how much longer stays in flight, it permanently fails to encounter other robots at the communication links. In this case, we say that the robot \emph{starves} and  the system falls in \emph{starvation state} if all robots starve. Figure~\ref{fig:starvation_intro} shows a synchronized system that falls in starvation state when two robots leave the team, note that the surviving robots $u_1$ and $u_3$ permanently fail to encounter other robot in the link positions. Also, in Figure~\ref{fig:starvation_intro}, note that none of the trajectories is completely \emph{covered}, that is, every trajectory has a section that is not visited by the surviving robots. The covering of the trajectories and the concept of starvation in a multi-agent system leads to interesting problems in Discrete Mathematics and has applications in the study of the fault tolerance of a synchronized system.

\begin{figure}[!h]
\begin{subfigure}{.245\textwidth}
\centering
\includegraphics{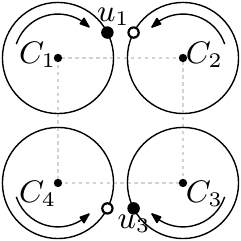}
\caption{}
\label{fig:starvation_intro_a}
\end{subfigure}%
\begin{subfigure}{.245\textwidth}
\centering
\includegraphics[page=2]{starvation_steps_intro.pdf}
\caption{}
\label{fig:starvation_intro_b}
\end{subfigure}%
\begin{subfigure}{.245\textwidth}
\centering
\includegraphics[page=3]{starvation_steps_intro.pdf}
\caption{}
\label{fig:starvation_intro_c}
\end{subfigure} %
\begin{subfigure}{.245\textwidth}
\centering
\includegraphics[page=4]{starvation_steps_intro.pdf}
\caption{}
\end{subfigure}
\caption{When the robots represented by non-solid points leave the system then the surviving robots, solid points, follow the paths drawn with fat stroke.}
\label{fig:starvation_intro}
\end{figure}

\subsection{Contribution and Organization}

In this work, we introduce the notion of resilience in a synchronized multi robot system, and consider two global tasks to be maintained after possible fails of some robots in the system. The first goal is to avoid the starvation phenomenon and the second one to ensure all trajectories are covered. The resilience is defined as maximum number of robots that can leave the team so that the system does not fall in starvation state or no trajectory is abandoned. The larger its value the more fault tolerant the system is.
We present some general properties of the synchronized systems and specified properties of the resilience concept that help to calculate the resilience for usual communication graphs as trees, cycles or grids.

This paper is organized as follows: the next section states the theoretical model, presents some necessary concepts and results that have already been proven in a previous work \cite{dbanez2015icra}. Section~\ref{ch:res} presents formally the problem of the resilience. We introduce a notion of \emph{ring} and study their properties in Section~\ref{ch:rings}. Sections~\ref{sec:ucov_res} and \ref{ch:isol} contain our results on the uncovering-resilience and the isolation-resilience, respectively.
Finally, Section~\ref{ch:conclusions} presents the conclusion of this work and some future research.

\section{Preliminaries}
Before introducing the problem studied in this paper, we formally state the framework and show some useful properties of a synchronized system. In our study, we focus on a simple model. However, the results can be generalized to more realistic scenarios following the results of \cite{dbanez2015icra}.

In this simple model, all robots move in equal and pairwise disjoint circular trajectories at the same (constant) speed with the same communication range. Therefore, every agent takes the same time to make tour in a trajectory, this time is called \emph{system period}. For simplicity, we work with unit circles representing the fly trajectories of the $n$ robots and let $r$ be the communication range of the robots. Also, the following remark is considered:

\begin{remark}\label{rem:systemPeriod_basicTime}
One unit of time corresponds to one system period.
\end{remark}

\begin{definition}[System of robots]
A \emph{system of robots} is a triplet $(T,U,r)$ where $U=\{u_1, \dots,u_n\}$ is a set of identical robots, $T=\{C_1,\dots,C_n\}$ is a set of pairwise disjoint unit circles (trajectories) and $r$ is the communication range of the robots. 
\end{definition}

\begin{figure}[h]
    \centering
    \includegraphics[width=.4\textwidth, trim = 4cm 5cm 4cm 6cm, clip=true]{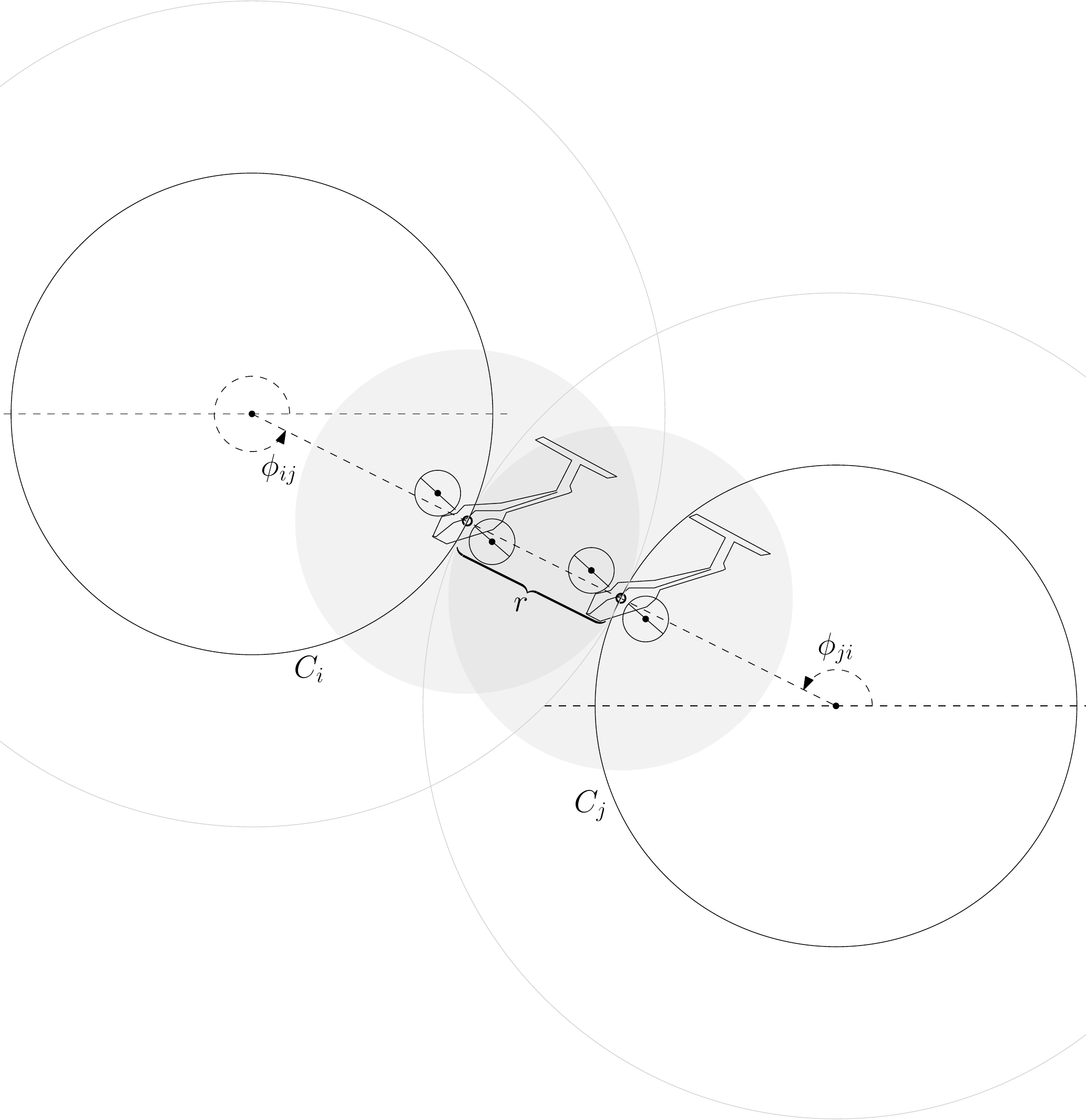}
    \caption{The simple model. Robots $i$ and $j$ can share information if the distance between $C_i$ and $C_j$ is less or equal to $r$. $\phi_{ij}$ (resp. $\phi_{ji}$) is the angle at which $i$ (resp. $j$) is closest to $j$'s trajectory (resp. $i$'s trajectory).}
    \label{fig:robots}
\end{figure}


\begin{definition}[Communication Graph]
Let $(T,U,r)$ be a system of robots. The \emph{communication graph} of the system is the graph $G=(V,E)$ whose nodes are the elements of $T$, that is $V=T$, and two nodes are connected by an edge if and only if the distance between their centers is less or equal to $2+r$ (see Figure~\ref{fig:2angles}). We denote by $(i,j)$ the edge that connect the adjacent nodes $C_i$ and $C_j$.
\end{definition}

\begin{figure}[h]
    \centering
    \includegraphics[width=.45\textwidth]{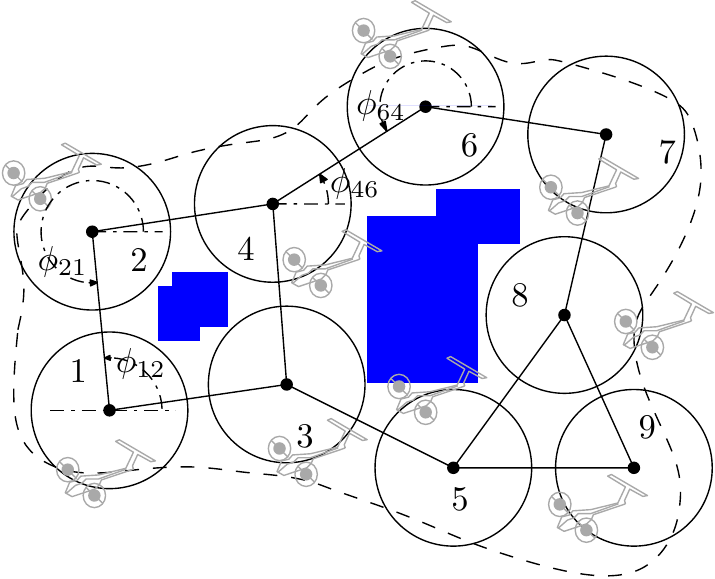}
    \caption{Representation of a set $C_1,C_2,\dots,C_9$ of unit circles and the underlying communication graph.}
    \label{fig:2angles}
\end{figure}

In the analysis that follows, it will be convenient to denote the position of an robot by the angle, measured from the positive horizontal axis, that its position makes on the unit circle (see Figure~\ref{fig:robots}). 

\begin{remark}\label{rem:positions_congruence}
The position of a robot in a trajectory is an angle in $[0;2\pi)$. Therefore, if $p$ is a position then $p=p+2k\pi$, for all $k\in\mathbb{Z}$.
\end{remark}

\begin{definition}[Link Position]\label{def:link_position}
Let $G=(V,E)$ be the communication graph of a system of robots and $(i,j)$ an arbitrary edge of $E$. The \emph{link position of $C_i$ with respect to $C_j$}, denoted by $\phi_{ij}$, is the angle of the point in $C_i$ that is closest to $C_j$.
\end{definition}

\begin{remark}
If $\phi_{ij}$ is defined, so is $\phi_{ji}$ and $\phi_{ji}=\phi_{ij}+\pi$ (equivalently $\phi_{ji}=\phi_{ij}-\pi$).
\end{remark}

\begin{definition}[Synchronized Neighbors]
Two robots flying in $C_i$ and $C_j$ respectively, are \emph{neighbors} if $(i,j)\in E$. Then, two neighbors flying in $C_i$ and $C_j$ respectively are \emph{synchronized} if both arrive at the same time to the link positions $\phi_{ij}$ and $\phi_{ji}$, respectively. Notice that the number of synchronized pairs is bounded by $|E|$. If the number of synchronized pairs is $|E|$, we say that the system of robots is \emph{synchronized}.
\end{definition}

\begin{remark}\label{rmk:repeated_meetings}
Because of our model assumptions, two synchronized neighbors ``meet" each other repeatedly every $2\pi$ units of travel. Equivalently, two synchronized neighbors ``meet" each other repeatedly every one unit of time. 
\end{remark}

Note that, since the robots move at constant speed, it suffices to know the starting position and movement direction (clockwise or counterclockwise) of an robot in order to compute its position at any time. 

\begin{definition}[Schedule]
Let $T$ be the set of trajectories of a system of robots. A \emph{schedule} of the system is a pair $F=(f,g)$ of functions $$f:T\rightarrow [0;2\pi),$$ $$g:T\rightarrow\{1;-1\},$$ such that for all $C_i\in T$ the value $f(C_i)\in [0;2\pi)$ is the starting position for a robot in $C_i$ and $g(C_i)\in\{1,-1\}$ indicates the movement direction, 1 is counterclockwise (CCW) and -1 is clockwise (CW). The position of the robot in the circle $C_i$ at the time $t$ is 
\begin{equation} \label{sync}
f(C_i)+g(C_i)\cdot 2\pi \cdot t.
\end{equation}
\end{definition}

Attending to the above definitions the following remark is deduced:
\begin{remark}\label{rem:synchro1}
Let $G=(V,E)$ and $F=(f,g)$ be the communication graph and a schedule of a system of robots, respectively. Let $(i,j)$ be an edge of $E$, the robots flying in $C_i$ and $C_j$ are synchronized neighbors if and only if 
$$\forall t\geq 0\quad f(C_i)+g(C_i)\cdot 2\pi\cdot t=\phi_{ij} \Longleftrightarrow f(C_j)+g(C_j)\cdot 2\pi\cdot t=\phi_{ji}.$$
\end{remark}

Because of our model assumptions, the previous remark can be rewritten as follows:
\begin{remark}\label{rem:synchro2}
Let $G=(V,E)$ and $F=(f,g)$ be the communication graph and a schedule of a system of robots, respectively. Let $(i,j)$ be an edge of $E$, the robots flying in $C_i$ and $C_j$ are synchronized neighbors if and only if 
$$\displaystyle\exists t_0\in[0;1)\;\text{such that}\; f(C_i)+g(C_i)\cdot 2\pi\cdot t_0=\phi_{ij} \text{\; and\;} f(C_j)+g(C_j)\cdot 2\pi\cdot t_0=\phi_{ji},$$ moreover,
$$\forall k\in \mathbb{N}\quad f(C_i)+g(C_i)\cdot 2\pi\cdot (t_0+k)=\phi_{ij} \text{\quad and\quad} f(C_j)+g(C_j)\cdot 2\pi\cdot (t_0+k)=\phi_{ji}.$$
\end{remark}

\begin{definition}[Synchronization Schedule]\label{def:synchro_schedule}
Let $G=(V,E)$ and $F=(f,g)$ be the communication graph and a schedule of a system of robots, respectively. $F$ is a \emph{synchronization schedule} if the system is synchronized. 
\end{definition}

This definition can be formulated using the remarks \ref{rem:synchro1} and \ref{rem:synchro2} as follows:
\begin{remark}
Let $G=(V,E)$ and $F=(f,g)$ be the communication graph and a schedule of a system of robots, respectively. $F$ is a synchronization schedule if and only if $\forall (i,j)\in E$ it holds that
$$\forall t\geq 0\quad f(C_i)+g(C_i)\cdot 2\pi\cdot t=\phi_{ij} \Longleftrightarrow f(C_j)+g(C_j)\cdot 2\pi\cdot t=\phi_{ji},$$
or equivalently,
$$\exists t_0\in[0;1)\text{\;such that\;} f(C_i)+g(C_i)\cdot 2\pi\cdot t_0=\phi_{ij} \text{\; and \;} f(C_j)+g(C_j)\cdot 2\pi\cdot t_0=\phi_{ji}.$$
\end{remark}

We are interested in systems of robots where neighboring robots are flying in opposite directions due to this scheme provides certain degree of robustness (see Subsection III.C of \cite{dbanez2015icra}). Formally, let $G=(V,E)$ be the communication graph of a system of robots and let $F=(f,g)$ be a schedule on it, then $g(C_i)=-g(C_j)$ for all $(i,j)\in E$. Therefore, a schedule where every pair of neighbors are flying in opposite directions can be established only if the communication graph is bipartite.

The next results were proven in \cite{dbanez2015icra}:
\begin{lemma}[Lemma 3.2 of \cite{dbanez2015icra}]\label{lem:synchro}
Let $G=(V,E)$ be the communication graph of a system of robots.
Let $F=(f,g)$ be a schedule where every pair of neighbors are flying in opposite directions.
Let $(i,j)$ be an edge in $E$. The robots in $C_i$ and $C_j$ are synchronized neighbors if and only if $$f(C_j) = 2 \beta_{ij} - f(C_i) \pm \pi.$$
\end{lemma}

$\beta_{ij}$ is the angle measured from the positive horizontal axis of the supporting line of the edge $(i,j)$, see Figure~\ref{fig:proof-diff-dir}. 

\begin{figure}[h]
\centering
\includegraphics[scale=.7]{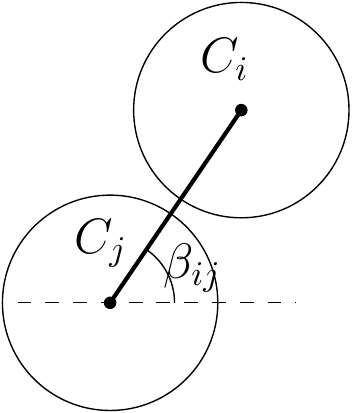}
\caption{Angle of the edge $(i,j)$ measured from the positive horizontal axis.}
\label{fig:proof-diff-dir}
\end{figure}

We deduced the following result from Lemma~\ref{lem:synchro}: 

\begin{corollary}\label{cor:adj_opp}
Let $G=(V,E)$ be the communication graph of a system of robots.
Let $F=(f,g)$ be a schedule where every pair of neighbors are flying in opposite directions.
The system is synchronized if and only if 
$$f(C_j) = 2 \beta_{ij} - f(C_i) \pm \pi \quad \forall (i,j)\in E.$$
\end{corollary}

The following theorem provides necessary and sufficient conditions for the existence of a schedule that synchronizes a system of robots with every pair of neighbors flying in opposite directions.

\begin{theorem}[Theorem 3.2 of \cite{dbanez2015icra}]\label{thm:synchro}
Let $G=(V,E)$ be the communication graph of a system of robots.
Then there exists a synchronization schedule for the system with every pair of neighbors flying in opposite directions if and only if $G$ is bipartite and every cycle
$\langle C_{i_1},C_{i_2},\ldots,C_{i_{2k}},C_{i_1}\rangle$ in $G$  
satisfy:
$$\beta_{i_1i_2}-\beta_{i_2i_3}+\beta_{i_3i_4}-\beta_{i_4i_5}+\dots+\beta_{i_{2k-1}i_{2k}}-\beta_{i_{2k}i_1}=2m\pi,\quad m\in\mathbb{N}.$$
\end{theorem}

Now, let us consider the case in which a robot leaves the system and then the assigned task is interrupted. In this case, in order to keep the performance of all the tasks maintaining the assigned trajectories, the following simple strategy is proposed in \cite{dbanez2015icra}. Let $(T,U,r)$ be a system of robots using a synchronization schedule $F=(f,g)$ under the conditions of Theorem~\ref{thm:synchro}. Let $u$ be a robot flying in a trajectory $C_i\in T$. If $u$ arrives to the link position $\phi_{ij}$ and the neighbor in $C_j$ has left the system, then $u$ moves to $C_j$ and follows the direction $g(C_j)$ (movement direction assigned to $C_j$, see Figure~\ref{fig:ch_routes}).

In order to keep the synchronization scheme, it is assumed that a robot can accelerate when it is passing to another trajectory such that the time taken to pass from a trajectory to another is negligible. Therefore, in the studied model it is assumed the following statement:

\begin{statement}\label{stm:swapping}
In a system of robots using a synchronization schedule, if a robot $u$, flying in $C_i$, arrives to the link position $\phi_{ij}$ at time $t$ and the neighbor in $C_j$ has left the system, then $u$ passes to $C_j$ instantly and follows the schedule assigned to $C_j$. That is, for any small time interval $\varepsilon>0$, at time $t+\varepsilon$, the robot $u$ lies in $C_j$ at the position  $$f(C_j)+g(C_j)\cdot 2\pi \cdot (t+\varepsilon).$$ 
\end{statement}

The following result is proven in \cite{dbanez2015icra}: 

\begin{theorem}[Theorem 3.3 of \cite{dbanez2015icra}]\label{thm:no_2r_in_traject}
In a synchronized system of robots under the previous statement, each trajectory is occupied at most by one robot.
\end{theorem}

From now on, we assume a system of robots whose communication graph is connected. Otherwise, we study every connected component separately. Also, we are assuming that the communication graph fulfills Theorem~\ref{thm:synchro}, if it does not, we extract the maximum subgraph fulfilling the synchronization conditions of Theorem~\ref{thm:synchro}.

\begin{definition}[Synchronized Communication System (SCS)] A \emph{synchronized communication system (SCS)} is a synchronized system of robots with neighbors flying in opposite directions under the Statement~\ref{stm:swapping}.
\end{definition}

With above definitions and results we are ready to formally introduce the resilience problem in the next section. 



\section{Resilience Measures} \label{ch:res}

We focus on the study of a SCS where some robots leave the monitoring group of robots (could be at different times). We do not assume that all the remaining robots are starving, but we assume that the remaining robots do not leave the system. We call this system a \emph{partial SCS}. Note that the notion of partial SCS generalizes the notion SCS since the number of leaving robots could be zero.

First, we introduce the concept of \emph{starvation}.
\begin{definition}[Starvation]
In a partial SCS, we say that a robot $u$ \emph{starves} if for all $ (i,j)\in E$ every time that $u$ arrives to $\phi_{ij}$ the circle $C_j$ is empty. We say that the system is in \emph{starvation state} if all the robots starve.
\end{definition}

See Figure~\ref{fig:starv_samples}. In the shown samples if the white robots leave the team then the system falls in starvation state. 

\begin{figure}[ht]
\centering
\begin{subfigure}{0.3\textwidth}
\centering
\includegraphics[page=1]{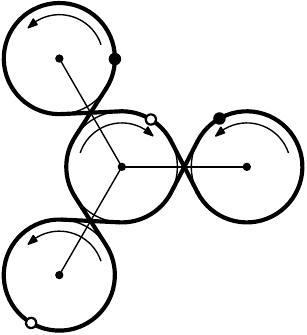}
\caption{}
\label{fig:tree_bound1}
\end{subfigure}%
\begin{subfigure}{0.3\textwidth}
\centering
\includegraphics[page=2]{starvation.pdf}
\caption{}
\label{fig:tree_bound2}
\end{subfigure} %
\begin{subfigure}{0.3\textwidth}
\centering
\includegraphics[page=3]{starvation.pdf}
\caption{}
\label{fig:tree_bound3}
\end{subfigure} 

\begin{subfigure}{0.49\textwidth}
\centering
\includegraphics[page=5]{starvation.pdf}
\caption{}
\label{fig:no_tree_bound1}
\end{subfigure} %
\begin{subfigure}{0.49\textwidth}
\centering
\includegraphics[page=4]{starvation.pdf}
\caption{}
\label{fig:no_tree_bound2}
\end{subfigure}
\caption{In these examples if the robots represented with non-solid points leave the team then the system falls in starvation state. The paths of the surviving robots, solid points, are represented using fat stroke.}
\label{fig:starv_samples}
\end{figure}



\begin{definition}[Covered trajectory]
In a partial SCS, we say that a trajectory is \emph{covered} if all the points forming it are visited periodically. Analogously, a trajectory is \emph{uncovered} if it has a section that is not visited by the surviving robots.
\end{definition}

Note that a trajectory is covered if it is surrounded by the paths of the surviving robots, see the first three samples in Figure~\ref{fig:starv_samples}. All the trajectories in the samples of figures \ref{fig:no_tree_bound1}, \ref{fig:no_tree_bound2} and \ref{fig:starvation_intro} are uncovered.

\begin{definition}[Isolation-Resilience]
In a SCS, the \emph{isolation-resilience} of the system, is the maximum number of robots that can fail without resulting in starvation of the system.
\end{definition}



\begin{definition}[Uncovering-Resilience]
In a SCS, the \emph{uncovering-resilience} of the system, is the maximum number of robots that can fail such that all trajectories are covered by the surviving robots.
\end{definition}

Clearly, the isolation-resilience (uncovering-resilience) with $n>1$ robots is an integer in $[0,n-2]$ ($[0,n-1]$).

In this paper, we explore properties of both tolerance measures and show their values for some special configurations.

\section{Rings and Partial SCS} \label{ch:rings}
In this section we introduce a notion of a \emph{ring} and show some properties of rings. First, we study the motion of a starving robot. 

\subsection{Starving Robots and Rings}
A starving robot in a system is always passing from a trajectory to another at link positions. 
The analysis of motion of starving robots in a SCS will help to study the resilience of the system.
Figure~\ref{fig:ring_sample} shows examples of the motion of starving robots. 

\begin{figure}[htb]
\centering
\includegraphics{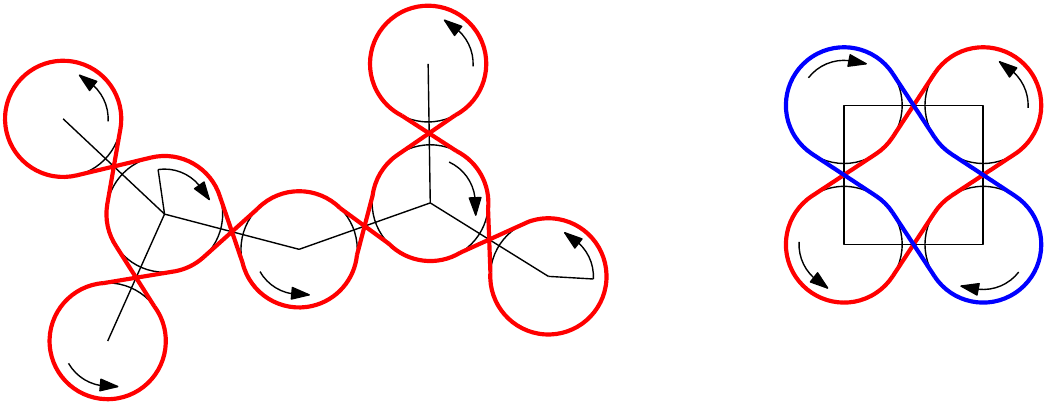}
\caption{The motion of a starving robot (a) in a tree and (b) in a cycle graph. Notice that the path of the robot in (a) covers all the trajectories whereas, in (b), there are two different motions of a starving robot.}
\label{fig:ring_sample}
\end{figure}

We first introduce an auxiliary graph called the \emph{starving-motion graph} to model the motion of a starving robot. 

\begin{definition}[Starving-motion graph (SMG)]
Let $T$ be a set of trajectories, let $G=(V,E)$ be a communication graph on $T$ and let $g$ be a direction assignment such that $g(C_i)=-g(C_j)$ for any $(i,j)\in E$. We construct the \emph{starving-motion graph (SMG)} as follows:

\begin{enumerate}
\item for each trajectory $C_i\in T$ enumerate the link points in clockwise order,
\item for each link position, add two nodes to the SMG and label them by $v^{(i)}_k$ and $w^{(i)}_k$, where $i$ is the index of the trajectory and $k$ the corresponding index in the previous numeration,
\item if $C_i$ has $m$ link positions, then 
\begin{itemize}
\item if $g(C_i)=1$ add an edge $\left(w^{(i)}_k,v^{(i)}_{k-1}\right)$ for all $1< k \leq m$, and add the edge $\left(w^{(i)}_1,v^{(i)}_m\right)$,
\item if $g(C_i)=-1$ add an edge $\left(v^{(i)}_k,w^{(i)}_{k+1}\right)$ for all $1\leq k < m$, and add the edge $\left(v^{(i)}_m,w^{(i)}_1\right)$, 
\end{itemize}
\item and finally, for any edge $(i,j)\in E$ between the $k$-th link position in $C_i$ and the $l$-th link position in $C_j$, 
\begin{itemize}
\item if $g(C_i)=1$, add the edges $\left(v^{(i)}_k,v^{(j)}_l\right)$ and $\left(w^{(j)}_l,w^{(i)}_k\right)$,
\item if $g(C_i)=-1$, add the edges $\left(v^{(j)}_l,v^{(i)}_k\right)$ and $\left(w^{(i)}_k,w^{(j)}_l\right)$.
\end{itemize}

\end{enumerate}

\end{definition}

\begin{figure}[h]
\begin{subfigure}[b]{.4\textwidth}
\centering
\includegraphics[scale=0.4]{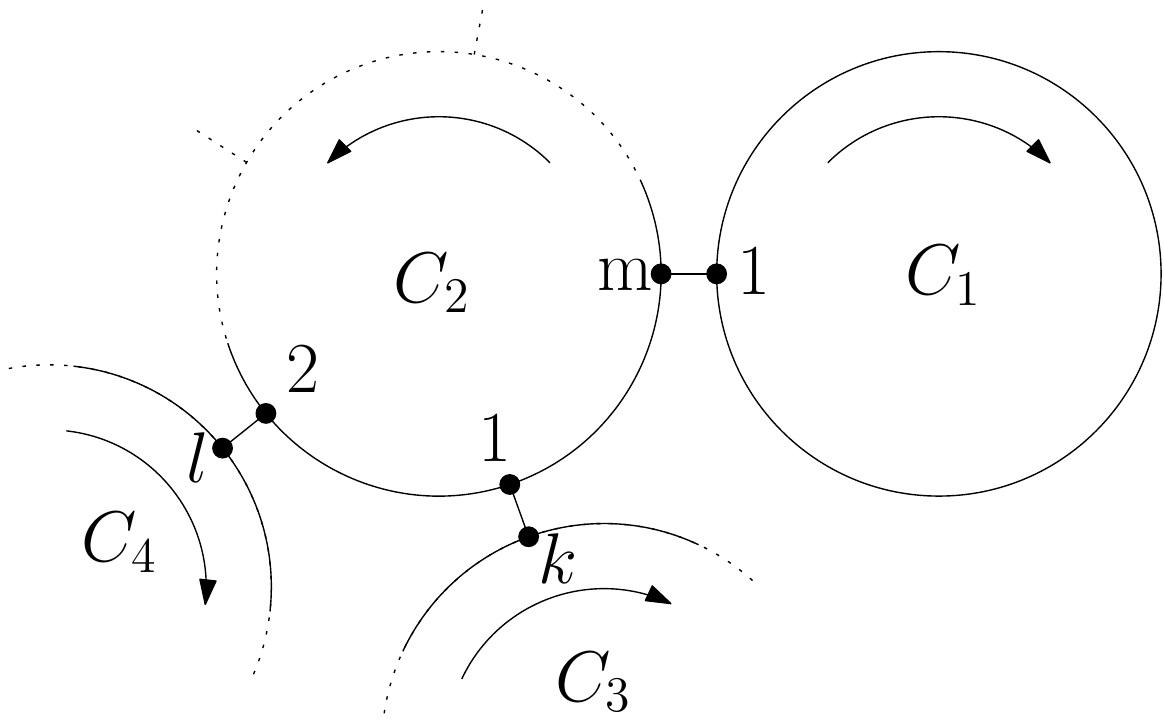}
\caption{}
\end{subfigure}%
\begin{subfigure}[b]{.6\textwidth}
\centering
\includegraphics[scale=0.8, page=2]{rings_characterization.pdf}
\caption{}
\end{subfigure}
\caption{A section of $G$ (a) and the corresponding section in the SMG (b).}
\label{fig:linked_graph}
\end{figure}

Figure~\ref{fig:linked_graph} illustrates the SMG construction. Note that the in-degree and the out-degree of every node in the SMG is exactly 1.

Let $H$ be the SMG of a SCS. Suppose that after some robots have left the team, a robot $u$ starves. The motion of $u$ in the resultant system can be modeled by a path of nodes in $H$. Suppose that the robot $u$ starts to fly in $C_i$ and $g(C_i)=1$ (resp. $g(C_i)=-1$). Obviously it will visit a link position, that is a $v$-node (resp. $w$-node) in $H$, from here, $u$ crosses to a link position in other trajectory $C_j$ where $g(C_j)=-1$ (resp. $g(C_j)=1$), that is, it passes through an $(v,v)$-edge (resp. $(w,w)$-edge) in $H$ to a $v$-node (resp. $w$-node). From here, following the direction of motion, the robot reaches other link position in $C_j$ (can be the same if $C_j$ has just one link position). From the last visited $v$-node (resp. $w$-node) using a $(v,w)$-edge (resp. $(w,v)$-node), it passes to a $w$-node (resp. $v$-node), and so on. Thus, the motion of a starving robot in a system can be represented as a sequence of nodes in the corresponding SMG.\\


\begin{lemma}\label{lem:closed_rings}
The motion of a starving robot always follows a closed path.
\end{lemma}

\begin{proof}
Note that the SMG is the union of directed cycles since
every node has both in-degree and out-degree one. 
A starving robot follows one of the cycles in the SMG.
\comment{
Let $H$ be the SMG corresponding to the system. Let $x_1,x_2,\dots$ be the infinite sequence of nodes in $H$ that represents the path of a starving robot moving indefinitely in the system. Since $H$ has two nodes for each link position (the number of link positions is finite), this sequence repeats nodes. Let $x_i$ be the first repeated node in the sequence, that is, $x_i$ is the first node such that there exists $x_j=x_i$ with $j>i$. Suppose $i>1$, using that the in-degree of $x_i$ is 1, we have that $x_{i-1}=x_{j-1}$, in fact, a contradiction is obtained, $x_i$ is not the first repeated node. From here we have that $i=1$. Using that the out-degree of every node in $H$ is 1, we have that $x_2=x_{j+1}$, $x_3=x_{j+2}$, and so on. We deduce that the sequence is of the form $x_1,\dots,x_k,x_1,\dots,x_k,x_1,\dots$, therefore, a ring is a closed path.
}
\end{proof}

\begin{definition}[Ring]
A \emph{ring} is the closed path in the plane corresponding to the motion of a starving robot in a system. 
\end{definition}

We have the following remarks:
\begin{remark}\label{rem:ring_trajectory_decomposition}
Every arc (between two consecutive link points) of a trajectory $C_i$ belongs to a ring.
\end{remark}

\begin{remark}\label{rem:one_direction}
Every ring has a unique direction (determined by the direction of corresponding trajectories).
\end{remark}

\begin{remark}\label{rem:ring_sequence}
A ring can be represented by a cycle in the corresponding SMG.
\end{remark}

Two \emph{rings are disjoint} if they do not use any arc of positive length in any trajectory. That is, two disjoint rings may only intersect at link positions. 
The remarks above and the fact that the SMG is the union of directed cycles imply the following lemma. 

\begin{lemma}\label{lem:disjoint_rings}
Let $T$ be a set of trajectories, let $G=(V,E)$ be a communication graph on $T$ and let $g$ be an assignation of motion directions such that $g(C_i)=-g(C_j)$ for all $(i,j)\in E$. Then
\begin{itemize} 
\item
The union of all trajectories in $T$ is a collection of pairwise disjoint rings for $(T,G,g)$.
\item 
The SMG for $(T,G,g)$ is the union of edge-disjoint cycles.
\item The set of rings for $(T,G,g)$ is in the one-to-one correspondence with the set cycles in the SMG for $(T,G,g)$.
\end{itemize}
\end{lemma}

\comment{
\begin{proof}
Assume that two different rings share sections of trajectories. Let $H$ be the corresponding SMG. Let $X=\{x_1,\dots,x_k\}$ and $Y=\{y_1,\dots,y_l\}$ be the corresponding circular sequences in $H$ of each ring. If the rings are not disjoint, then there exists a common continuous subsequence in $X$ and $Y$. Without loss of generality, assume that $k\leq l$ and the first $r$ nodes of $X$ and $Y$ form the common continuous subsequence with maximum length. If $r<k$, using that out-degree of $x_r$ is 1 and $y_r=x_r$ we obtain that $x_{r+1}=y_{r+1}$ and this is a contradiction because we are assuming that the common continuous sequence has length $r$. If $r=k$ and $k<l$ then using the previous argument we obtain that $x_1=y_{r+1}=y_1$ and this is a contradiction because all the nodes in $Y$ should be different. The only remaining option is $r=k=l$, that is, $X$ and $Y$ are the same circular sequence.
\end{proof}

From lemmas \ref{lem:disjoint_rings} and \ref{lem:closed_rings}, and remarks \ref{rem:one_direction} and \ref{rem:ring_sequence} we deduce the following result:
\begin{corollary}
Let $T$ be a set of trajectories, let $G=(V,E)$ be the communication graph on $T$ and let $g$ be an assignation of motion directions such that $\forall (i,j)\in E\;g(C_i)=-g(C_j)$. Let $H$ be the SMG generated by $(T, G, g)$. The connected components of $H$ are cycle graphs, and each connected component in $H$ is a ring.
\end{corollary}
}

Note that we can change the direction for each trajectiory $C_i\in T$. How does it affect the rings?

\begin{lemma}[Opposite direction] \label{lem:same_ring_collection}
Let $T$ be a set of trajectories, let $G=(V,E)$ be the communication graph on $T$ and let $g$ be a direction assignment such that $g(C_i)=-g(C_j)$ for all $(i,j)\in E$. The collection of rings  generated by $(T, G, -g)$ is the same as the collection of rings generated by $(T,G,g)$ but the starving robots move in opposite direction.
\end{lemma}

\begin{proof}
Consider the SMG graph for the system $(T,G,g)$.  
By the definition of starving-motion graph (SMG), in order to construct the SMG for $(T,G,-g)$, 
\begin{itemize} 
\item
we use the same set of vertices, even the same type $v$ or $w$; 
\item 
we enumerate the vertices of each trajectory $C_i$ differently (in opposite direction);
\item 
we flip the direction of each edge in the MSG.
\end{itemize}
Therefore, each ring in the SMG for $(T,G,-g)$ corresponds to a ring in the SMG for $(T,G,g)$ but the direction of the ring is reversed.
\end{proof}

\subsection{Partial Synchronized Communication Systems}

In the rest of this section we assume that some robots left the monitoring group of robots. We show some properties of partial SCSs and their rings.

From Lemma~\ref{lem:disjoint_rings} we have that every point in any trajectory belongs to only one ring. 
The robots are identified with their location, so, at an arbitrary instant of the mission we say that a robot is in a ring if it is located in a point of the ring.

We focus now on the number of robots in a ring.

\begin{lemma}\label{lem:invariant}
In a partial SCS, the number of robots in a ring is the same at any time.
\end{lemma}

\begin{proof}
The number of robots in a ring may change only at link positions. Consider a link between two trajectories $C_i$ and $C_j$. Suppose that they correspond to two different rings, say $R_i$ and $R_j$. Suppose that a robot arrives to a link position between $C_i$ and $C_j$. If its neighbor arrives too, then they keep their trajectories and pass each to the other ring, see Figure~\ref{fig:cycle_invariant_with_neighbor} (a). Otherwise, the robot passes to the neighboring trajectory and keeps the ring, see Figure~\ref{fig:cycle_invariant_with_no_neighbor} (b). In any case the number of robots in the rings $R_i$ and $R_j$ does not change. 

If the trajectories $C_i$ and $C_j$ belong to the same ring (as in Figure~\ref{fig:ring_sample} (a)), we can apply the same argument.
\end{proof}

\begin{figure}[h]

\begin{subfigure}{.5\textwidth}
\centering
\includegraphics[scale=.8, page=2]{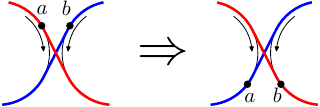}
\caption{}
\label{fig:cycle_invariant_with_neighbor}
\end{subfigure}%
\begin{subfigure}{.5\textwidth}
\centering
\includegraphics[scale=.8, page=3]{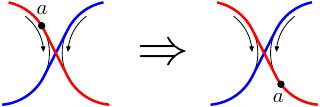}
\caption{}
\label{fig:cycle_invariant_with_no_neighbor}
\end{subfigure}
\caption{Invariant number of robots in a ring.}
\end{figure}

\begin{definition}[Length of a ring]\label{def:length_ring}
Considering that the time spent by a robot to cross from a trajectory to another is negligible (Statement \ref{stm:swapping}), the \emph{length of a ring} is defined as the sum of the lengths of the arcs of the trajectories forming the ring. Analogously, the \emph{length of a section of a ring} is defined as the sum of the lengths of the arcs of the trajectories forming the section of the ring.
\end{definition}

Figure~\ref{fig:illustrating_length} illustrates the above definition.
\begin{figure}[h!]
\centering
\includegraphics[page=3]{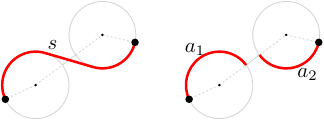}
\caption{The length of the section $s$ of a ring is the sum of the lengths of the arcs $a_1$ and $a_2$.}
\label{fig:illustrating_length}
\end{figure}

\begin{lemma}\label{lem:robot-ring_visit_circles}
Let $p$ be a point in a partial SCS visited by a robot at some time.
Then $p$ is visited by robots periodically. 
Furthermore, the period is equal to $l/(2\pi)$ where $l$ is the length of the ring containing $p$. 
\end{lemma}

\begin{proof}
Let $a$ be a robot in a ring $r$ at an arbitrary instant of time. Let $A_1, \dots, A_m$ be the arcs between consecutive link points traversed by $r$, they are enumerated following the direction of motion in $r$. This is a circular ordering, then $A_1$ follows $A_m$ and so on. Let $A_i$ be the arc where $a$ is. From this instant $a$ visits all the points of $A_i$ between its current position and the next link point (end of the arc). When $a$ arrives to the link position, if the neighboring trajectory has a robot $b$ then they meet each other in the link position and $b$ starts to travel the arc $A_{i+1}$, see Figure~\ref{fig:cycle_invariant_with_neighbor}. Otherwise, if the neighboring trajectory is empty when $a$ arrives to the link position then $a$ passes to the neighboring trajectory and starts to travel the arc $A_{i+1}$. Therefore, if an arc is being visited by a robot then the next arc will be visited by the same robot or by the neighbor. Applying this argument successively, it is deduced that all the arcs will be covered following the motion direction into the ring. The lemma follows. 
\end{proof}

As a consequence of Remark~\ref{rem:ring_trajectory_decomposition} and lemmas \ref{lem:invariant} and \ref{lem:robot-ring_visit_circles} we obtain the following result:
\begin{corollary}\label{cor:not-covered-circle}
In a partial SCS, a trajectory is not covered if and only if has a section of a ring with no robots.
\end{corollary}

From this corollary and Remark~\ref{rem:ring_trajectory_decomposition} we deduce the following result which is an important tool to determine the uncovering-resilience of a SCS.
\begin{theorem}\label{thm:not-covered-circle}
In a partial SCS, every trajectory is covered if and only if, at some time, every ring has a robot.
\end{theorem}

We prove an auxiliary claim for a path between two robots in a ring, see Figure~\ref{fig:dist_2_neighbor} for an example.
\begin{figure}[ht]
\centering
\includegraphics[scale=1]{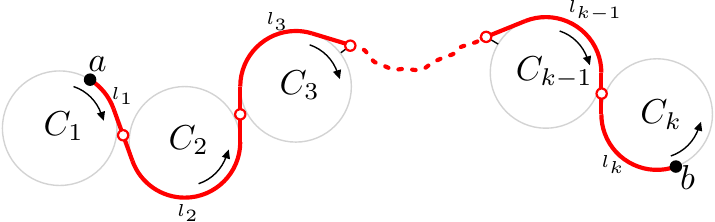}
\caption{A path between two robots in a ring.}
\label{fig:dist_2_neighbor}
\end{figure}

\begin{lemma} \label{aux1}
Let $\sigma$ be a path between two robots in a ring in a partial SCS at time $t$.
If $\sigma=A_1A_2\dots A_k$ where $A_i,i=1,2\dots,k$ is a directed arc of $C_i$, 
then, for  all $1\leq j\le k$, 
\begin{equation}\label{eq:phi}
f(C_j)+g(C_j)\cdot2\pi\cdot(t+\sum_{i=1}^j t_i)=\theta_{j}.
\end{equation}
where $t_i$ is the required time by a robot to traverse $A_i$ and 
$\theta_{j}$ is the angle of the second endpoint of $A_j$. 
\end{lemma}

\begin{proof}
We prove Equation (\ref{eq:phi}) by induction on $j$. 
For $j=1$, by Equation (\ref{sync}), we have the base case
$$f(C_1)+g(C_1)\cdot 2\pi\cdot (t+t_1) = \theta_{1}.$$
{\em Induction step}.  Suppose that (\ref{eq:phi}) holds for some $j<k$. 
By the definition of synchronization schedule we have
$$f(C_{j+1})+g(C_{j+1})\cdot 2\pi\cdot \left(t+\sum_{i=1}^{j} t_i\right)=
\theta'_{j+1},$$ 
where $\theta'_{j+1}$ is the position of the first endpoint of $A_{j+1}$. 
Since $\theta_{j+1}=\theta'_{j+1}+g(C_{j+1})\cdot 2\pi \cdot t_{j+1}$, we have
$$f(C_{j+1})+g(C_{j+1})\cdot 2\pi\cdot \left(t+\sum_{i=1}^{j+1} t_i\right)=\theta_{j+1},$$
and the lemma follows.
\end{proof}

\begin{lemma} \label{lem:robots_distance}
In a partial SCS, the length of a path between any two robots in the same ring is in $2\pi\mathbb{N}$.
\end{lemma}

\begin{proof}
Let $a$ and $b$ be the positions of two robots in the same ring at time $t$ and let $\sigma$ be the path in the ring from $a$ to $b$ as denoted in Lemma \ref{aux1} and shown in Figure~\ref{fig:dist_2_neighbor}. Note that $C_i$ and $C_{j}$ (with $1\leq i<j\leq k$) could be the same trajectory, see Figure~\ref{fig:twice_same_trajectory}. Thus, if $F=(f,g)$ is the synchronization schedule on the system we have 

\begin{eqnarray}
f(C_1)+g(C_1)\cdot 2\pi\cdot t = a\label{eq:start_a}\\
f(C_k)+g(C_k)\cdot 2\pi\cdot t = b\label{eq:end_b}
\end{eqnarray}

\begin{figure}[ht]
\centering
\includegraphics[scale=1, page=2]{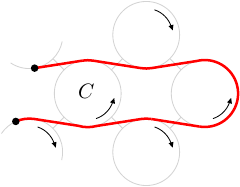}
\caption{A path between two robots in a ring that has two arcs in the trajectory $C$.}
\label{fig:twice_same_trajectory}
\end{figure}

Then
\begin{align*}
f(C_{k})+g(C_{k})\cdot 2\pi\cdot\left(t+\sum_{i=1}^{k}t_i\right) &=b
\text{\; (by Lemma \ref{aux1})}\\
f(C_{k})+g(C_{k})\cdot 2\pi\cdot t+g(C_{k})\cdot 2\pi\sum_{i=1}^{k}t_i&=b\\
b+g(C_{k})\cdot 2\pi\sum_{i=1}^{k}t_i&=b \text{\; (by Equation (\ref{eq:end_b}))}\\
g(C_{k})\cdot 2\pi\sum_{i=1}^{k}t_i&=0
\end{align*}

Therefore the angle $2\pi\sum_{i=1}^{k}t_i$ is in $2\pi\mathbb{Z}$. 
Since $2\pi\sum_{i=1}^{k}t_i$ is also the the length of the path $\sigma$, the lemma follows.
\end{proof}

\begin{lemma}
The length of every ring in a SCS is in $2m\pi$ for some $m\in \mathbb{N}$.
\end{lemma}

\begin{proof}
Let $r$ be a ring. If $r$ consists of only one trajectory, then its length is $2\pi$. If $r$ consists of arcs of more than one trajectory. Let $p$ be a position in the ring, then the ring can be viewed as a path from $p$ to $p$ in the ring $r$.  
By Lemma \ref{lem:robots_distance} the length of the ring is in $2\pi\mathbb{N}$. 
\end{proof}

\begin{lemma}\label{lem:ringCapacity}
In a SCS, a ring with length $2k\pi$ has at most $k$ robots.
Also, if no robots have left the team, then a ring with length $2k\pi$ has exactly $k$ robots.
\end{lemma}
\begin{proof}
In a SCS trajectories with $n$  the total length of the trajectories is $2n\pi$. Let $m$ be the number of rings in the system and let $2k_1\pi,2k_2\pi,\dots,2k_m\pi$ be their lengths. By Lemma~\ref{lem:disjoint_rings} we obtain that $k_1+k_2+\dots+k_m=n$. Then, from Lemma~\ref{lem:robots_distance} we deduce that the $i$-th ring has at most $k_i$ robots for all $1\leq i \leq m$. 

In a SCS, where no robots have left the team, there are $n$ robots, one per trajectory. Let $r_i$ be number of robots in the $i$-th ring. We have that $r_i\leq k_i$ and $\displaystyle\sum_{i=1}^mk_i=\displaystyle\sum_{i=1}^mr_i=n$, then from this we deduce that $r_i=k_i$ for all $1\leq i\leq m$.
\end{proof}


\section{The Uncovering-Resilience}\label{sec:ucov_res}

This section is dedicated to study some properties of the uncovering-resilience that can be used to compute its value in a SCS.

\begin{lemma}\label{lem:cres_min_ring}
In a SCS the uncovering-resilience is $k-1$ where $k$ is the minimum of robots within a ring of the system.
\end{lemma}
\begin{proof}
Let $m$ be the number of rings of the system and let $k_1,\dots,k_m$ be the initial number of robots within the rings. 
The number of robots in every ring is constant unless a robot from it leaves the system (Lemma~\ref{lem:invariant}). If, at some time, a robot $u$ from a ring $r$ leaves the system then the number of robots in $r$ decreases by one and the other rings keep their number of robots.
Therefore, to empty the $i$-th ring it is necessary that $k_i$ robots from that ring leave the system. Then, the minimum number of robots that must leave the system to obtain an empty ring is $k=\min\{k_1,\dots,k_m\}$. Finally, using Theorem~\ref{thm:not-covered-circle} the result follows.
\end{proof}

Using the previous lemma the uncovering-resilience of a SCS can be easily computed using a simple algorithm.
Since an alive robot describes a ring the approach is as follows. We take one robot, delete its ring and the robots inside the ring and continue with another alive robot. In this way, we can calculate both the rings in the system and the number of robots in each ring. Also, Lemma~\ref{lem:ringCapacity} relates the length of a ring with the initial number of robot on it, so determining the rings and their lengths we can compute the uncovering-resilience of the SCS.


\begin{theorem}
Let $G$ be the communication graph of a system. Let $F_1=(f_1,g_1)$ and $F_2=(f_2,g_2)$ be two different synchronization schedules on $G$. The SCS on $G$ using $F_1$ has the same uncovering-resilience as the SCS on $G$ using $F_2$.
\end{theorem}
\begin{proof}
Suppose that the communication graph $G$ is connected. 
If $g_1=g_2$ then the rings in a SCS using $F_1$ and the rings in other one using $F_2$ are the same. If $g_1=-g_2$ then the rings are the same but using opposite directions of movement (Lemma~\ref{lem:same_ring_collection}). In any case the number of rings and their lengths are the same. Therefore, by Lemma~\ref{lem:ringCapacity}, the initial number of robots within the every ring is the same regardless the chosen synchronization schedule. By Lemma~\ref{lem:cres_min_ring} both synchronized communication system have the same uncovering-resilience.

If the communication graph $G$ is not connected, the uncovering-resilience of a SCS is equal to the smallest uncovering-resilience of a SCS corresponding to a connected component of $G$. By the same argument as for the connected graph $G$, the theorem follows. 
\end{proof}

The previous theorem prove that the uncovering-resilience of a SCS only depends on $T$ and $G$, 
that is, given trajectories and a communication graph, the uncovering-resilience of every SCS on it will be the same regardless of the chosen synchronization schedule. This lead us to present the following definition:
\begin{definition}[Uncovering-resilience of a communication graph]
Let $T$ be the set of trajectories and $G$ be the communication graph of a system. The \emph{uncovering-resilience of $(T,G)$} denoted by $\UcovRes(T,G)$ is the maximum number of robots that can fail in a SCS on $G$ such that all the trajectories are covered by the surviving robots.
\end{definition}

The following subsections address the study of some properties in specific communication graphs that allow us to determine a priori the value of the uncovering-resilience without executing the above algorithm.

\subsection{Uncovering-Resilience of Trees}
We start analyzing the simplest graphs, the trees.
\begin{lemma}\label{lem:tree-single-ring}
If the communication graph of a system is a tree then there is a single ring.
\end{lemma}

\begin{proof} We prove the lemma by induction on the number of trajectories. Clearly, if there is only one trajectory, the ring is unique.

\begin{figure}[htb]
\centering
\includegraphics{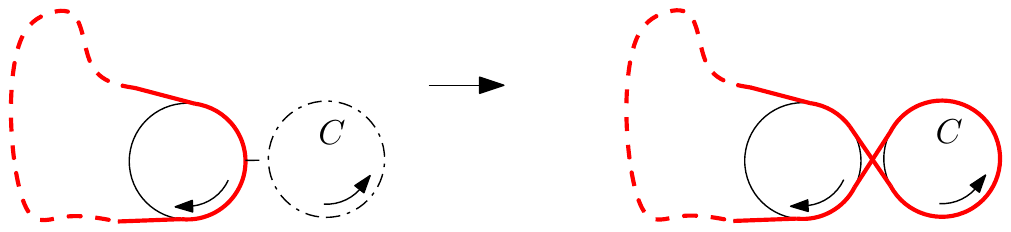}
\caption{(a) Ring corresponding to $T'$.
(b) Ring corresponding to $T$.}
\label{f2}
\end{figure}

Suppose that the claim holds for any tree with $n$ trajectories. We show that it also holds for any tree $T$ with $n+1$ trajectories. Let $C$ be a trajectory corresponding to a leaf in $T$, see Figure~\ref{f2} (a). Let $T'$ be the tree obtained by deleting trajectory $C$. Then
there is exactly one ring corresponding to $T'$. Adding $C$ to the system changes the ring by adding a loop covering $C$ as shown in Figure~\ref{f2} (b). The lemma follows.
\end{proof}

From Lemma~\ref{lem:cres_min_ring}, it is easy to see that if the communication graph has a single ring then the uncovering-resilience of the graph is $n-1$, where $n$ is the number of trajectories in the system. This observation and the previous lemma derive directly in the next result:

\begin{corollary}
If the communication graph of a system with $n$ trajectories is a tree then its uncovering-resilience is $n-1$.
\end{corollary}

\subsection{Uncovering-Resilience of Grids}
We say that the communication graph of a system has a grid configuration if there are $n\cdot m$ trajectories distributed en $n$ rows and $m$ columns, and we call it an \emph{$n\times m$ grid communication graph} or simply a \emph{grid communication graph}. The trajectories in a same row are aligned horizontally and the trajectories in the same column are aligned vertically. Each trajectory is identified by a pair $(i,j)$ where $i$ and $j$ indicate the row and the column where the trajectory is, respectively. In a grid communication graph the trajectory $(i,j)$ is linked (using edges) with the trajectories $(i-1,j)$, $(i,j-1)$, $(i+1,j)$ and $(i,j+1)$ if it exists, see Figure~\ref{fig:grid_sample}. 

\begin{figure}[h!]
\centering
\includegraphics[scale=0.5]{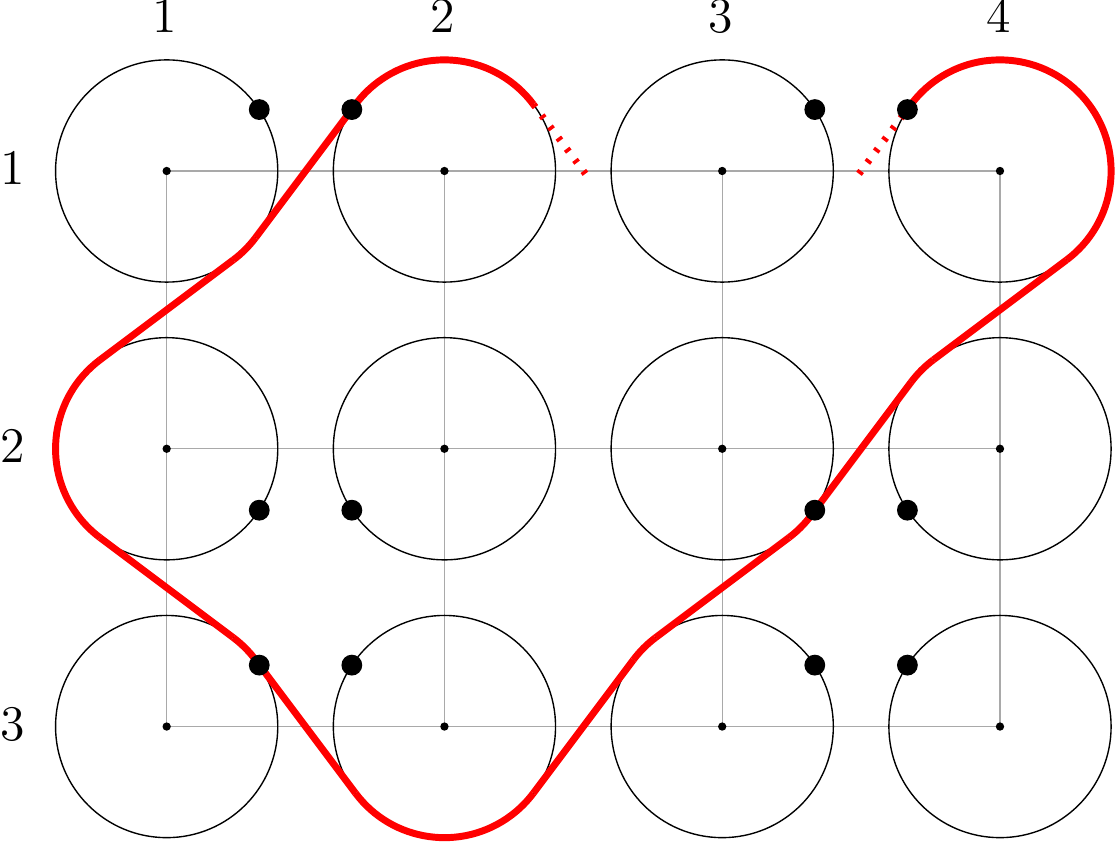}
\caption{This is a $3\times 4$ grid communication graph. The drawn section of a ring hits the top wall at the trajectories $(1,2)$ and $(1,4)$, it hits the bottom wall at $(3,2)$, the left wall at $(2,1)$ and the right wall at $(1,4)$.}
\label{fig:grid_sample}
\end{figure}

In a grid graph we say that a ring \emph{hits} the top wall of the grid if the ring pass through the top section of a circle in the first row. Analogously we can define when a ring hits the left, bottom or right wall of the ring. See Figure~\ref{fig:grid_sample}.

\begin{figure}[h!]
\begin{subfigure}[b]{.5\textwidth}
\centering
\includegraphics[scale=1]{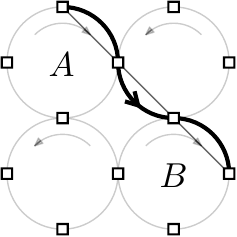}
\caption{}
\end{subfigure}%
\begin{subfigure}[b]{.5\textwidth}
\centering
\includegraphics[scale=1,page=2]{diag-move.pdf}
\caption{}
\end{subfigure}

\begin{subfigure}[b]{.5\textwidth}
\centering
\includegraphics[scale=1,page=6]{diag-move.pdf}
\caption{}
\end{subfigure}%
\begin{subfigure}[b]{.5\textwidth}
\centering
\includegraphics[scale=1,page=5]{diag-move.pdf}
\caption{}
\end{subfigure}%
\caption{Diagonal movement of a starving robot in a grid communication graph.}
\label{fig:diag-move}
\end{figure}

Figure~\ref{fig:diag-move} shows the movement of a starving robot in different sections of a ring between the circles. In this figure, with small white squares are noted two kind of points, the points where the ring hits the walls and the contact points between two neighboring circles. Note that these points are visited diagonally (with slope $1$ and $-1$) by a starving robot.

\begin{figure}[h!]
\centering
\includegraphics[scale=1]{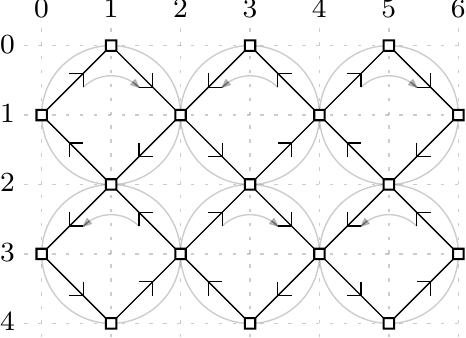}
\caption{Simplified diagram of movement in a grid communication graph.}
\label{fig:diag-diagram}
\end{figure}

Let $G$ be a $n\times m$ grid communication graph. We construct a simplified diagram of movement on $G$ using the hitting points, the contact points between the neighboring circles and the segments between them. This diagram has $2n+1$ and $2m+1$ rows and columns respectively. The rows are indexed from 0 to $2n$ and the columns are indexed from 0 to $2m$, see Figure~\ref{fig:diag-diagram}.

\begin{lemma}\label{lem:nxn_rings}
A system with an $n\times n$ grid communication graph has $n$ rings with length $2n\pi$ and each ring hits the four walls exactly once.
\end{lemma}
\begin{proof}
Let $H$ be the simplified diagram of movement on an $n\times n$ grid communication graph. Let $r$ be the ring that hits the top wall at the vertex $(0,i)$ of $H$. It is easy to see that $r$ hits the right wall at the vertex $(2n-i,2n)$, the bottom wall at the vertex $(2n, 2n-i)$ and the left wall in $(i,0)$. Also, $r$ does not hit the walls in any other vertex. Note that a step from a vertex to other in the simplified diagram of movement corresponds to a section of a ring with length $\pi/2$. The ring $r$ visits $4n$ vertices on $H$, thus the length of $r$ is $4n\cdot \pi/2$, that is $2n\pi$. Repeating this argument for each hitting point in the top wall we obtain $n$ rings of length $2n\pi$. The sum of the lengths of all circles in the system is $2n^2\pi$ and the sum of the lengths of the $n$ rings in the system is $2n^2\pi$ too, so there is no other ring in the system.
\end{proof}

Let $G_1$ and $G_2$ be two grid communication graphs of size $n\times m$ and $n\times p$ respectively. The concatenation of $G_1$ and $G_2$, such that the last trajectory in the $i$-th row of $G_1$ is linked with the first trajectory in the $i$-th row of $G_2$ (for all $1\leq i\leq n$), produces a new $n\times (m+p)$ grid communication graph, see Figure~\ref{fig:h-operation}.


\begin{figure}[h]
\centering
\includegraphics[width=\textwidth]{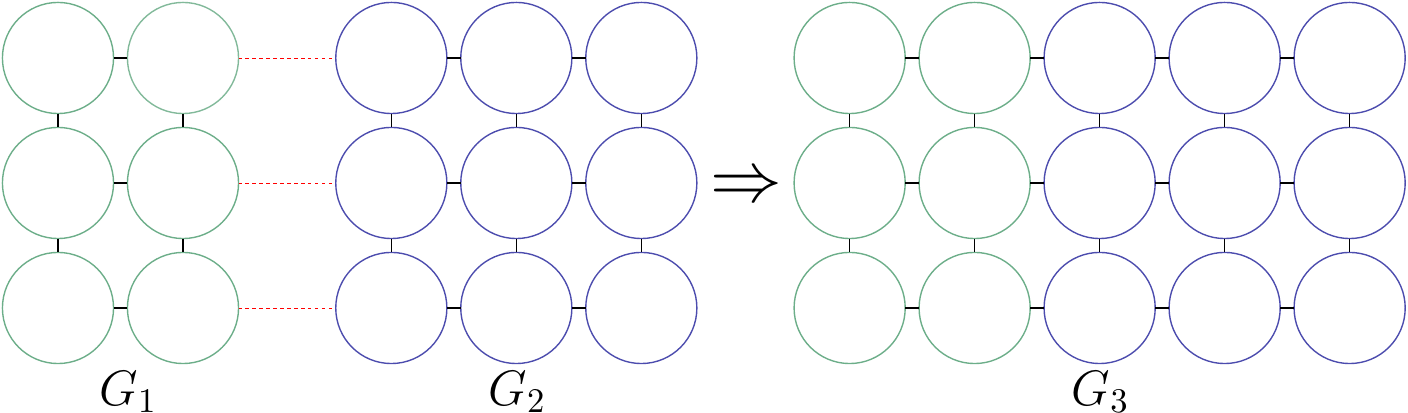}
\caption{$G_3$ is the resultant graph of the concatenation of $G_1$ and $G_2$.}
\label{fig:h-operation}
\end{figure}

The following result is a technical lemma that we'll need to complete the proof of Theorem~\ref{thm:gcd_rings_grid}.
\begin{lemma}\label{lem:hooking-rings}
Let $G$ be an $n\times m$ grid communication graph with $k$ rings of the same length $l$, and every ring in $G$ hits the left and right walls exactly $c$ times, and hits the top and bottom walls exactly $c'$ times. The resultant grid graph of the concatenation of $G$ and an $n\times n$ grid graph has $k$ rings of the same length $l+2cn\pi$ and every ring in the resultant graph hits the left and right walls exactly $c$ times, and hits the top and bottom walls exactly $c'+c$ times. 
\end{lemma}
\begin{proof}
Let $G$ be an $n\times m$ grid communication graph with $k$ rings of the same length $l$, and every ring in $G$ hits the left and right walls exactly $c$ times, and hits the top and bottom walls exactly $c'$ times. Let $Q$ be an $n\times n$ grid communication graph. Let $R$ be the resultant graph of the concatenation of $G$ and $Q$. 

$Q$ has $n$ rings of length $2n\pi$, and all of them hit each wall once (Lemma~\ref{lem:nxn_rings}). Thus, every ring in $Q$ extends in $2n\pi$ the length of a ring in $G$, see Figure~\ref{fig:rings_fusion}. Since every ring in $G$ hits the right wall $c$ times then in the concatenation each ring in $G$ is fused with $c$ rings in $Q$. Note that two rings in $G$ can not be fused between them in the concatenation. Then, every ring in $R$ is formed by the fusion of a ring in $G$ and $c$ rings in $Q$. Therefore, $R$ has $k$ rings of the same length $l+2cn\pi$. Every obtained ring hits the left wall $c$ times in the $c$ hitting points of the respective ring in $G$, and hits the right wall $c$ times in the hitting points of the $c$ respective rings in $Q$. Every ring in $R$ hits the top and bottom walls $c'+c$ times, in the $c'$ hitting points of the respective ring in $G$ and in the hitting points of the $c$ respective rings in $Q$.
\begin{figure}
\centering
\includegraphics[page=2, width=\textwidth]{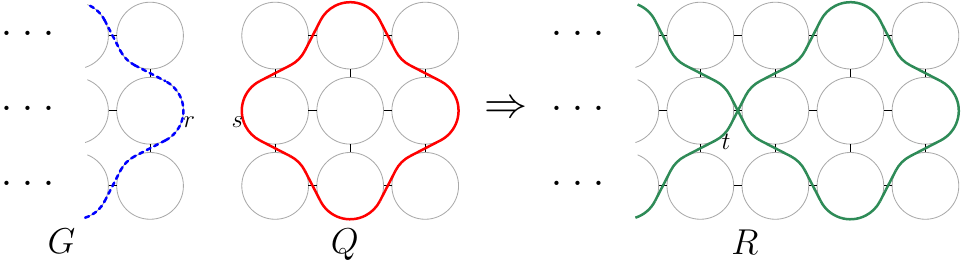}
\caption{$r$ is the ring in $G$ that hits the right wall in the second row. $s$ is the ring in $Q$ that hits the left wall in the second row. $t$ is the ring in $R$ obtained from $r$ and $s$ in the concatenation.}
\label{fig:rings_fusion}
\end{figure}
\end{proof}

The following result is the key of this subsection.
\begin{theorem}\label{thm:gcd_rings_grid}
An $n\times m$ grid communication graph has $gcd(n,m)$ (greatest common divisor of $n$ and $m$) rings of the same length, and all of them have the same number of hitting points in each wall.
\end{theorem}
\begin{proof}
We prove the result by induction in the number of rows. For $n=1$, we have that every $1\times m$ grid communication graph has a single ring, $gcd(1,m)=1$, one hitting point in left and right walls, and $m$ hitting points in the top and bottom walls. Then, the theorem holds for these cases. Assume as inductive hypothesis that: for a fixed value $N$, the theorem holds for every $n\times m$ grid communication graph with $n\leq N$. 

Let prove the theorem for $(N+1)\times m$ grid communication graphs. If $m\leq N$ then, using that an $(N+1)\times m$ grid is equivalent to a $m\times (N+1)$ grid, the theorem holds by inductive hypothesis. If $m=N+1$ then the theorem holds by Lemma~\ref{lem:nxn_rings}. In order to prove the theorem for a $(N+1)\times m$ grid communication graph with $m>N+1$ we use induction in the number of columns. Assume as second inductive hypothesis that: for a fixed value $M\geq N+1$, the theorem holds for every $(N+1)\times m$ grid communication graph with $m\leq M$. 

Let $G$ be an $(N+1)\times (M+1)$ grid communication graph. We have that $M+1>N+1$, then removing the last $N+1$ columns of $G$ we obtain an $(N+1)\times (M-N)$ grid communication graph $G'$. The theorem holds for $G'$ by the second inductive hypothesis. The $N+1$ removed columns conform an $(N+1)\times (N+1)$ grid communication graph. Concatenating $G'$ with the $N+1$ removed columns we obtain $G$. Then, using that the theorem holds for $G'$, Lemma~\ref{lem:hooking-rings} and properties of $gcd(\cdotp,\cdotp)$ the result follows.

\end{proof}

The concluding result of this subsection is:
\begin{theorem}
The uncovering-resilience of an $n\times m$ grid communication graph is: $$\frac{n\cdot m}{gcd(n,m)}-1.$$
\end{theorem}
\begin{proof}
The initial number of robots in this system is $n\cdot m$. Using Theorem~\ref{thm:gcd_rings_grid} we have that the minimum number of robots within a ring in the system is $n\cdot m/gcd(n,m)$. 
Using Lemma~\ref{lem:cres_min_ring} the theorem follows.
\end{proof}

\section{The Isolation-Resilience} \label{ch:isol}

Analogously to Section~\ref{sec:ucov_res}, this one is dedicated to study the isolation-resilience. Unfortunately, for this measure we do not know a polynomial algorithm to compute its value for an arbitrary SCS. However, we show some properties in specific cases that allow us to compute the value of the isolation-resilience or at least establish some bounds for its value.

We start introducing a formula that relates the value of the isolation-resilience of a system with the maximum possible number of starving robots.
\begin{lemma}\label{lem:resilience_rel_starv}
In a SCS of $n$ robots, let $\IsolRes$ be the isolation-resilience of the system and let $\starv$ be the maximum possible number of starving robots in the system. It holds that 
$$\starv+R = n-1.$$
\end{lemma}
\begin{proof}
By the definition of resilience we have that $n-R>\starv$. Using that $\starv$ is the maximum possible number of starving robots in the system, then removing $R+1$ robots of the SCS it is possible that the system falls in starvation state, thus $n-R-1\leq \starv$. Due to  $\starv$ is an integer value, we obtain that $n-R-1=\starv$ and the result follows.
\end{proof}

From this lemma and the definition of isolation-resilience it is deduced the following result:
\begin{corollary}\label{cor:max_starv}
The maximum possible number of starving robots in a SCS is reached only if the system is into starvation. 
\end{corollary}

The following four lemmas are technical results to proof the Lemma~\ref{lem:starv_any_synchro}.
\begin{lemma}\label{lem:opposite_synchro}
Let $G=(V,E)$ be the communication graph of a system. If $F=(f,g)$ is a synchronization schedule then $F'=(f,-g)$ is a synchronization schedule too.
\end{lemma}
\begin{proof}
Let $(i,j)$ be an arbitrary edge of $G$, then by Definition~\ref{def:synchro_schedule} there exists $t\in[0,1)$ such that:
$$f(C_i)+g(C_i)\cdot 2\pi \cdot t = f(C_j)+g(C_j)\cdot 2\pi\cdot t + \pi.$$

From this and using Remark~\ref{rem:positions_congruence} we obtain:
$$(f(C_i)+g(C_i)\cdot 2\pi \cdot t) - g(C_i)\cdot 2\pi = (f(C_j)+g(C_j)\cdot 2\pi\cdot t + \pi)-g(C_j)\cdot 2\pi.$$

The above equation can be rewritten as follows:
$$f(C_i)-g(C_i)\cdot 2\pi \cdot (1-t) = f(C_j)-g(C_j)\cdot 2\pi\cdot (1-t) + \pi.$$

If $0<t<1$ then $(1-t)\in [0,1)$. In this case, taking $t'=(1-t)$, we have that:
\begin{equation}\label{eq:t_greater_that_zero}
\exists t'\in[0;1)\qquad f(C_i)-g(C_i)\cdot 2\pi \cdot t' = f(C_j)-g(C_j)\cdot 2\pi\cdot t' + \pi.
\end{equation}

If $t=0$ then $1-t=1$ and $f(C_i)-g(C_i)\cdot 2\pi = f(C_j)-g(C_j)\cdot 2\pi + \pi.$ From this and using Remark~\ref{rem:positions_congruence} we obtain $f(C_i) = f(C_j) + \pi.$ In this case, taking $t'=0$, we have that:
\begin{equation}\label{eq:t_equal_to_zero}
\exists t'\in[0;1)\qquad f(C_i)-g(C_i)\cdot 2\pi \cdot t' = f(C_j)-g(C_j)\cdot 2\pi\cdot t' + \pi.
\end{equation}
Therefore, from (\ref{eq:t_greater_that_zero}) and (\ref{eq:t_equal_to_zero}) we deduce that $F' =(f,-g)$ is a synchronization schedule.
\end{proof}

\begin{lemma}\label{lem:matching_synchro}
Let $G=(V,E)$ be the communication graph of a system. Let $F_1=(f_1,g_1)$ and $F_2=(f_2,g_2)$ be two different synchronization schedules on the system. $F_1$ and $F_2$ reach a matching position every unit of time, that is, there exist times $t_1$ and $t_2$ such that $f_1(C_i)+g_1(C_i)\cdot 2\pi\cdot (t_1+t) = f_2(C_i)+g_2(C_i)\cdot 2\pi\cdot (t_2+t)$ for all $C_i\in V$ and $t\in \mathbb{N}$.
\end{lemma}
\begin{proof}
Obviously, there exist times $t_1$ and $t_2$ such that $F_1$ and $F_2$ match on $C_1$, that is:
\begin{equation}\label{eq:first_match}
f_1(C_1)+g_1(C_1)\cdot 2\pi\cdot t_1 = f_2(C_1)+g_2(C_1)\cdot 2\pi \cdot t_2.
\end{equation}
Let $C_i$ be a node adjacent to $C_1$ in $G$. We prove that $F_1$ and $F_2$ match on $C_i$ too. By using Lemma~\ref{lem:synchro}, we have that:
\begin{eqnarray}
f_1(C_i)+g_1(C_i)\cdot 2\pi\cdot t_1 = 2\beta_{1,i}-f_1(C_1)+\pi+g_1(C_i)\cdot 2\pi\cdot t_1\label{eq:f_1-C_i}\\
f_2(C_i)+g_2(C_i)\cdot 2\pi\cdot t_2 = 2\beta_{1,i}-f_2(C_1)+\pi+g_2(C_i)\cdot 2\pi\cdot t_2\label{eq:f_2-C_i}
\end{eqnarray}
Since $C_1$ and $C_i$ are neighbor, then $g_1(C_i)=-g_1(C_1)$ and $g_2(C_i)=-g_2(C_1)$. Replacing $g_1(C_i)$ by $-g_1(C_1)$ and $g_2(C_i)$ by $-g_2(C_1)$ in the right side of equations (\ref{eq:f_1-C_i}) and (\ref{eq:f_2-C_i}), respectively:
$$f_1(C_i)+g_1(C_i)\cdot 2\pi\cdot t_1 = 2\beta_{1,i}-f_1(C_1)+\pi-g_1(C_1)\cdot 2\pi\cdot t_1$$
$$f_2(C_i)+g_2(C_i)\cdot 2\pi\cdot t_2 = 2\beta_{1,i}-f_2(C_1)+\pi-g_2(C_1)\cdot 2\pi\cdot t_2.$$
Regrouping the terms and using (\ref{eq:first_match}) we obtain that 
$$f_1(C_i)+g_1(C_i)\cdot 2\pi\cdot t_1 = f_2(C_i)+g_2(C_i)\cdot 2\pi\cdot t_2,$$
thus, $F_1$ and $F_2$ match on $C_i$. Applying progressively this argument to every edge $(i,j)\in E$ such that $F_1$ and $F_2$ match on $C_i$, we deduce that $F_1$ and $F_2$ match on $C_j$ too. Using that $G$ is connected we derive 
$$\forall C_i\in V\quad f_1(C_i)+g_1(C_i)\cdot 2\pi\cdot t_1 = f_2(C_i)+g_2(C_i)\cdot 2\pi\cdot t_2.$$
Using that the robots make a tour in an unit of time the lemma follows.
\end{proof}


\begin{lemma}\label{lem:starv_g_-g}
Let $G=(V,E)$ and $F=(f,g)$ be the communication graph and synchronization schedule of a system in starvation, respectively. It is possible to obtain a system in starvation with the same number of robots using $F'=(f,-g)$ as synchronization schedule.
\end{lemma}
\begin{proof}
Let $A$ be a system in starvation with $k$ robots. Let $G=(V,E)$ and $F=(f,g)$ be the communication graph and synchronization schedule of $A$, respectively. $F'=(f,-g)$ is a synchronization schedule by Lemma~\ref{lem:opposite_synchro} and matches with $F$ by Lemma~\ref{lem:matching_synchro}. Take a SCS with no failures using $F'$. We prove that, at an instant $t_0$ when it is matching with $A$, if we remove exactly the robots lied in the trajectories that are empty in $A$, the resultant system, denoted by $A'$, also is in starvation. For seek a contradiction suppose that $A'$ is not in starvation.

Let $a$ and $b$ be the first pair of robot that exchanges information in $A'$. Let $(i,j)\in E$ be the edge where $a$ and $b$ establish the link communication. Let $t$ be the elapsed time from $t_0$ until the first meeting between $a$ and $b$ in $(i,j)$. That is, after $t$ units of time from $t_0$ the robots $a$ and $b$ arrive to $\phi_{ij}$ and $\phi_{ji}$ respectively. 

Let $r_a$ and $r_b$ be the rings of $a$ and $b$ respectively (may be the same ring). Let $s_a$ and $s_b$ be the positions of $a$ and $b$ at the instant $t_0$, respectively.
Let $2k\pi$ and $2k'\pi$ be the lengths of $r_a$ and $r_b$ respectively. The robot $a$ makes a complete travel in $r_a$ using $k$ units of time and $b$ makes a complete travel in $r_b$ using $k'$ units of time (Lemma~\ref{lem:same_ring_collection}). Let $t'\in \mathbb{N}$ be the lowest natural value such that $kk't'\geq t$. In $A$, after $kk't'$ units of time from $t_0$, the robots $a$ and $b$ are in $s_a$ and $s_b$ respectively. Then, since the rings are closed paths (Lemma~\ref{lem:closed_rings}), at time $kk't'-t$ the robots $a$ and $b$ are in $\phi_{ij}$ and $\phi_{ji}$ respectively. Therefore, they are not starving in $A$, contradicting the hypothesis. The result follows.
\end{proof}

\begin{lemma}\label{lem:same_starv_F1_F2_g}
Let $G=(V,E)$ be the communication graph of a system of robots. Let $F_1=(f_1, g)$ and $F_2=(f_2, g)$ be two synchronization schedules using the same movement directions. It is possible to obtain a system in starvation with the same number of robots using either of them ($F_1$ or $F_2$).
\end{lemma}
\begin{proof}
Let $A$ be a system in starvation with $k$ robots. Let $G=(V,E)$ and $F_1=(f_1,g)$ be the communication graph and synchronization schedule of $A$, respectively. $F_1$ and $F_2$ match by Lemma~\ref{lem:matching_synchro}. Let $A'$ be a SCS with no failures using $F_2$. When $A'$ reaches a matching position remove the robots lied in trajectories that are empty in $A$. Using that $A$ and $A'$ have the same assignment of movement directions, from this instant the behavior of the robots in $A'$ and $A$ are identical. The lemma follows. 
\end{proof}

With the previous results we are ready to introduce the following lemma:
\begin{lemma}\label{lem:starv_any_synchro}
The maximum possible number of starving robots of a system is the same using any synchronization schedule on it.
\end{lemma}
\begin{proof}
Let $F_1=(f_1, g_1)$ and $F_2=(f_2, g2)$ be two different synchronization schedules on the system. 
Let $k$ be the maximum possible number of robots in a system on $G$ using $F_1$, this value is reached when the system falls into starvation (Corollary~\ref{cor:max_starv}). If $g_1=g_2$ then using Lemma~\ref{lem:same_starv_F1_F2_g} the result follows. If $g_1=-g_2$ then in a SCS on $G$ using $F_1'=(f_1, -g_1)$ as synchronization schedule (Lemma~\ref{lem:opposite_synchro}) the maximum possible number of starving robots is $k$ by Lemma~\ref{lem:starv_g_-g}. Then, in a SCS on $G$ using $F_2$ as synchronization schedule the maximum possible number of starving robots also is $k$ by Lemma~\ref{lem:same_starv_F1_F2_g}.
\end{proof}

This lemma lead us to the following definition:
\begin{definition}[Starvation-number]
Let $G$ be the communication graph of a system of robots that fulfills Theorem~\ref{thm:synchro}. The \emph{starvation-number} of $G$, denoted by $\starv(G)$, is the maximum possible number of starving robots in a SCS on $G$.
\end{definition}

From lemmas \ref{lem:starv_any_synchro} and \ref{lem:resilience_rel_starv} we deduce the next result:
\begin{corollary}\label{cor:resilience_CW=CCW}
The isolation-resilience of a system of robots is the same using any synchronization schedule on it. 
\end{corollary}

\begin{definition}[Isolation-resilience (of a communication graph)]
Let $G$ be the communication graph of a system of robots that fulfills Theorem~\ref{thm:synchro}. The \emph{isolation-resilience} of $G$, denoted by $\IsolRes(G)$, is the maximum number of robots that can fail in a SCS on $G$ without resulting in starvation the system.
\end{definition}

From Lemma~\ref{lem:resilience_rel_starv} and the last two definitions we obtain:
\begin{remark}
Let $G$ be the communication graph of a system of $n$ robots that fulfills Theorem~\ref{thm:synchro}. It holds that:
$$\starv(G)+\IsolRes(G)=n-1.$$
\end{remark}


The isolation-resilience of a system can be calculated by computing the starvation-number of the communication graph. In this section we show how the isolation-resilience of specific cases can be bounded or exactly calculated.

\subsection{The Isolation-Resilience in Specific Cases}\label{ch:specific_cases}
In this section we show some results to compute the isolation-resilience of specific cases. In the first subsection we analyze the case in which the communication graph is a tree, then we compute the isolation-resilience of the cycle graphs and finally, the isolation-resilience of communication graphs with a grid configuration is obtained.

\subsubsection{Isolation-Resilience of Trees}
\begin{lemma}
The isolation-resilience of a simple chain of $n$ trajectories is $n-2$.
\end{lemma}
\begin{proof}
Let $G$ be the communication graph of a system, and it is a simple chain of $n$ trajectories, see Figure~\ref{fig:tree-chain} (the trajectories may not be on a straight line). To prove that $\IsolRes(G)=n-2$ it is enough to prove that $\starv(G)=1$. That is, we can not leave 2 robots in starvation in a SCS on $G$. 
\begin{figure}[ht]
\centering
\includegraphics[scale=0.7,page=2]{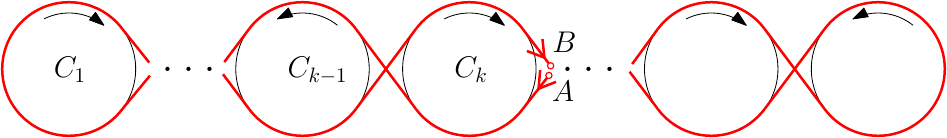}
\caption{The communication graph is a simple chain of trajectories.}
\label{fig:tree-chain}
\end{figure}

Suppose, for the sake of contradiction, that there exist 2 robots $p$ and $q$ in starvation in a SCS on $G$. Let $2k\pi$ be the length of the ring's section from $p$ to $q$ (following the motion direction of the ring), see Lemma~\ref{lem:robots_distance}. Note that $k<n$ because, otherwise $p$ and $q$ are the same robot. Let $C_k$ be the $k$-th trajectory from a leaf in $G$. In the cross between $C_k$ and $C_{k+1}$, let $A$ be the point where the ring enters in $C_k$, and let $B$ be the point where the ring leaves $C_k$, see Figure~\ref{fig:tree-chain}. The length of the ring's section from $A$ to $B$ is $2k\pi$ because it includes the $k$ first trajectories. Thus, when $p$ is entering in $C_k$ through $A$ then the robot $q$ is leaving $C_k$ through $B$. Therefore, $p$ and $q$ are in the cross between $C_k$ and $C_{k+1}$ at the same time, thus they are not starving.
\end{proof}

\begin{lemma} [Isolation-resilience of a tree]
Let $G$ be the communication graph of a system. If $G$ is a tree then $\starv(G)\leq \lfloor n/2 \rfloor$ and this bound is tight.
\end{lemma}
\begin{proof}
By Lemma \ref{lem:tree-single-ring}, there is only one ring in the SCS.
The length of the ring is $2n\pi$. Suppose, for the sake of contradiction, that there are at least $\lfloor n/2\rfloor+1$ starving robots. If the distance between every pair of consecutive robots in the ring is greater than or equal to $4\pi$ then the length of the ring is greater than or equal to $(\lfloor n/2\rfloor +1)4\pi$ and this is strictly greater than $2n\pi$. Then there must be two robots $p$ and $q$ such that $p$ is at distance $2\pi$ ahead of $q$. Consider the moment when $q$ enters a trajectory $C$ that corresponds to a leaf on $G$. In that moment $p$ is leaving the trajectory $C$ and meets $q$, thus they are not starving.

Examples with $\lfloor n/2\rfloor$ starving robots are shown in the figures \ref{fig:tree_bound1} and \ref{fig:tree_bound3}.
\end{proof}

Using the previous result and Lemma~\ref{lem:resilience_rel_starv} we obtain the following theorem:
\begin{theorem}
Let $G$ be the communication graph of a system. If $G$ is a tree then $\IsolRes(G)\geq\lceil n/2\rceil - 1$.
\end{theorem}

\subsubsection{Isolation-Resilience of Cycles}
The following results are about systems where the communication graph is a simple cycle.

\begin{figure}[h]
\centering
\begin{subfigure}[b]{.5\textwidth}
\centering
\includegraphics[scale=.9]{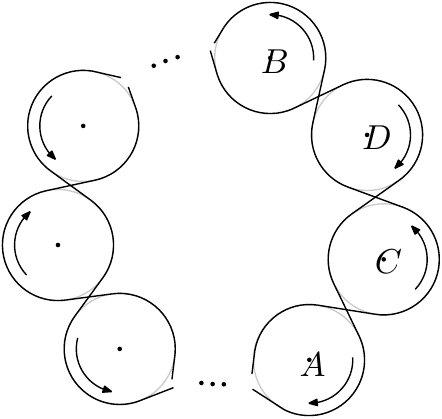}
\caption{}
\label{fig:cyc_tworings1}
\end{subfigure}%
\begin{subfigure}[b]{.5\textwidth}
\centering
\includegraphics[scale=.9,page=2]{cycle_rings.pdf}
\caption{}
\label{fig:cyc_tworings2}
\end{subfigure}

\begin{subfigure}[b]{.5\textwidth}
\centering
\includegraphics[scale=.9,page=3]{cycle_rings.pdf}
\caption{}
\label{fig:cyc_tworings3}
\end{subfigure}%
\begin{subfigure}[b]{.5\textwidth}
\centering
\includegraphics[scale=.9,page=4]{cycle_rings.pdf}
\caption{}
\label{fig:cyc_tworings4}
\end{subfigure}%
\caption{Induction proof of Lemma~\ref{lem:cyc_tworings}.}
\end{figure}

\begin{lemma}\label{lem:cyc_tworings}
Let $G$ be the communication graph of a system of robots. If $G$ is a cycle graph then it has two rings, one with CW direction and other with CCW direction. Also, every edge of $G$ corresponds to a cross between the two rings.
\end{lemma}
\begin{proof}
We proceed by induction on the number of trajectories in the system. We have that the cycle graphs are formed by an even number of trajectories, $2n$. For $n=2$, the cycle graph has 4 circles, see Figure~\ref{fig:ring_sample}(b), and the claim holds. Assume as inductive hypothesis that for a fixed value $n$ every cycle graph with $2n$ trajectories the claim holds. Now, let us consider a cycle graph $G$ with $2(n+1)$ trajectories (Figure~\ref{fig:cyc_tworings1}). If we remove two consecutive trajectories and stick the ends in the broken section we obtain a cycle graph $G'$ with $2n$ trajectories (Figure~\ref{fig:cyc_tworings2}). The claim holds for $G'$ by inductive hypothesis, Figure~\ref{fig:cyc_tworings3} shows the clockwise ring using normal stroke and the other using dashed stroke. The sections of rings in the removed two circles can be drawn like in Figure~\ref{fig:cyc_tworings3}. After that, inserting again the two removed trajectories into the original position we obtain the cycle graph $G$ (with $2(n+1)$ circles) holding the claim, see Figure~\ref{fig:cyc_tworings4}. Note that, as depicted in Figure~\ref{fig:cyc_tworings3}, when we insert the removed trajectories, the solid-circle point in the ring's section of the removed trajectories matches with the solid-circle point in the graph $G'$, analogously with the non-solid-circle points, the solid-square points and the non-solid-square points.
\end{proof}

%

\begin{lemma}\label{lem:starve_iff_same_cycle}
A set of robots in a cycle graph starves if and only if they belong to the same ring.
\end{lemma}
\begin{proof}
First we prove that if a set of robots starves then they belong to the same ring. We prove it by contradiction. Suppose that they do not belong to the same ring, then each ring in the cycle has at least one robot. In an arbitrary instant of time, select a robot from each ring, say  $u$ and $v$, respectively. One of them is flying in CW direction and the other in CCW direction (Lemma~\ref{lem:cyc_tworings}), therefore, after a while $u$ and $v$ will be in adjacent circles (trajectories) $C_i$ and $C_j$, respectively. By using Theorem~\ref{thm:no_2r_in_traject}, we deduce that it will never happen that $v$ cross to $C_i$ if $u$ is still there. Analogously, it will never happen that $u$ cross to $C_j$ if $v$ is still in $C_j$. Therefore, $u$ and $v$ simultaneously cross to $C_j$ and $C_i$ respectively. This is a contradiction because at that instant they meet each other and share information (no starve). As a consequence, they stay in theirs respective trajectories and do not cross.

If they are in the same ring then they starve. Suppose, on the contrary, that they do not starve, then there exist two robots that share info on an edge of the communication graph. From Lemma~\ref{lem:cyc_tworings} we have that every edge is a crossing between the two rings, thus if they share info then they are in different rings, a contradiction.
\end{proof}

\begin{remark}
For non-cycles graphs, this property is false as shown in Figure~\ref{fig:starving_two_cycles}.
\end{remark}

\begin{figure}[h]
\centering
\includegraphics[scale=1, page=4]{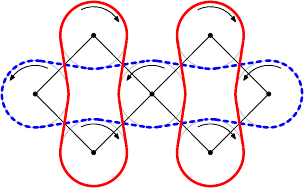}
\caption{There may be starving robots in the two rings drawn using continuous stroke.}
\label{fig:starving_two_cycles}
\end{figure}

Using the previous lemma we deduce that:
\begin{corollary}\label{cor:stn_cycles}
Let $a$ and $b$ the two rings of a system whose communication graph is a simple cycle. Let $N(a)$ and $N(b)$ be the number of robots of $a$ and $b$ respectively, then the starvation number of the graph is $\max\{N(a), N(b)\}$. 
\end{corollary}

With this result and using Lemma~\ref{lem:resilience_rel_starv} we conclude that:
\begin{theorem}\label{thm:cycle_graph_resilience}
Let $a$ and $b$ the two rings of a system whose communication graph is a simple cycle. Let $N(a)$ and $N(b)$ be the number of robots of $a$ and $b$ respectively. The isolation-resilience of the graph is $\min\{N(a),N(b)\}-1$.
\end{theorem}

\subsubsection{Isolation-Resilience of Grids}


We start showing how to compute the starvation number of gird communication graphs.

\begin{lemma}\label{lem:same_row_position}
Let $a$ and $b$ be two robots in circles that are in the same row in a grid communication graph. If the position of $a$ is $\alpha$, then the position of $b$ is $\alpha$ or $\pi-\alpha$ (see Figure~\ref{fig:same_row}).
\end{lemma}
\begin{proof}
Without loss of generality, assume that $a$ is to the left of $b$. Let $\beta$ be the position of $b$. We prove the lemma by induction on the number of circles between $a$ and $b$.

If $a$ and $b$ are in adjacent circles, then by using Lemma~\ref{cor:adj_opp} we obtain that $\beta=\pi-\alpha$.

If between $a$ and $b$ there is a circle, then let $\gamma$ be the position of the robot between $a$ and $b$, using Lemma~\ref{cor:adj_opp} we obtain that $\gamma=\pi-\alpha$. Again applying Lemma~\ref{cor:adj_opp} between $\gamma$ and $\beta$ we obtain that $\beta=\pi-\gamma=\pi-\pi+\alpha=\alpha$.

Assume as inductive hypothesis that for a fixed value $k\geq1$ if the number of robots between $a$ and $b$ is $k$ then $\beta=\alpha$ or $\beta=\pi-\alpha$.

Suppose now that the number of robots between $a$ and $b$ is $k+1$, let $\gamma$ be the position of the robot before $b$ in the row. By inductive hypothesis we have that $\gamma=\alpha$ or $\gamma=\pi-\alpha$. By using Lemma~\ref{cor:adj_opp} we obtain that $\beta=\gamma$ or $\beta=\pi-\gamma$, and replacing $\gamma$ by its value the result follows.
\end{proof}
\begin{figure}[h]
\centering
\includegraphics{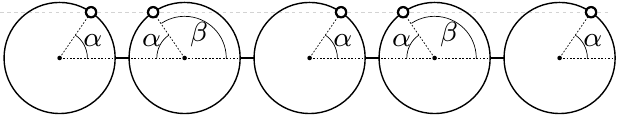}
\caption{The white points represent the positions of the robots in the same row of a grid communication graphs,  $\beta=\pi-\alpha$.}
\label{fig:same_row}
\end{figure}
\begin{lemma}\label{lem:same_row_always_same_row}
Suppose that at time $t_0$ there exist two starving robots in circles that are in the same row. Then, at any time after $t_0$ they will be in circles that are in the same row.
\end{lemma}
\begin{proof}
Suppose that $a$ and $b$ are in the same row, say the $i$-th row. Using the previous lemma, if the position of $a$ is $\alpha$ then the position of $b$ is $\alpha$ or $\pi-\alpha$. Using that the robots are moving at the same constant speed we have that, after a short time interval, the position of $a$ is $\alpha+\varDelta$ and the position of $b$ is $\alpha+\varDelta$ or $\pi-(\alpha+\varDelta)$ (still in the same row $i$). Thus, if $i>1$, when $a$ is at the position $\pi/2$, the position of $b$ is $\pi/2$ too. In that instant, they pass to the row $i-1$ and it holds the relation between their positions. Analogously, if $i$ is not the bottom row then when the position of $a$ is $3\pi/2$ the position of $b$ is $3\pi/2$ too, and they pass to the row $i+1$ and it holds the relation between their positions.
\end{proof}

Now, let us return to the diagram of movement of a grid communication graph, see Figure~\ref{fig:diag-diagram}. Let $G$ be a grid communication graph and let $H$ be its simplified diagram of movement. Two starving robots in circles that are in the same row in $G$ will reach hitting points or contact points with the same row in $H$. We focus on the study of the movement of these robots in the diagram $H$, that is, we look at their respective positions in $H$ every quarter of system period.

\begin{remark}\label{rem:keeping_row}
Let $a$ and $b$ be two robots in circles that are in the same row. If the position of $a$ in the diagram of movement is $(i,j)$, then, the position of $b$ in the diagram of movement is $(i,r)$ with $|r-j|$ even.
\end{remark}

There are four possible directions of movement for a starving robot in the simplified diagram of movement: up and left $(-1,-1)$, up and right $(-1,+1)$, down and left $(+1,-1)$, and down and right $(+1,+1)$, see Figure~\ref{fig:movement_dirs}.

\begin{figure}
\centering
\includegraphics[page=2]{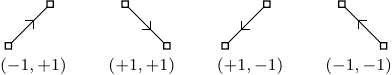}
\caption{The four possible movement direction in the simplified diagram of movement.}
\label{fig:movement_dirs}
\end{figure}

When a starving robot hits a wall then it changes its current movement direction. The following table shows the possible cases of direction changes.

\begin{table}[!h]
\centering
\renewcommand{\arraystretch}{1.3}
\begin{tabular}{|c||c|c|c|c|}
\hline
\diaghead{\theadfont \hspace*{2.35cm}}{Direction}{Hit wall} & Left & Top & Right & Bottom\\
\hline
\hline
$(-1,+1)$ & -- & $(+1,+1)$ & $(-1,-1)$ & -- \\
\hline
$(+1,+1)$ & -- & -- & $(+1,-1)$ & $(-1,+1)$\\
\hline
$(+1,-1)$ & $(+1,+1)$ & -- & -- & $(-1,-1)$\\
\hline
$(-1,-1)$ & $(-1,+1)$ & $(+1,-1)$ & -- & --\\ 
\hline
\end{tabular}
\caption{Changes of movement direction when a robot hits a wall.}
\label{tab:dirs_changes}
\end{table}

The leftmost column of the Table~\ref{tab:dirs_changes} (headers of the rows) contains the movement direction of the robot before it hits the wall. The top row (headers of the columns) contains the walls of the grid. The inner cells contains the direction of movement of the robot after it hits the wall indicated by the column header moving in the direction indicated by the row header. For example: if a robot is moving in direction $(-1,-1)$ and hits the left wall then it changes its new direction of movement is $(-1,+1)$, thus, the value of the cell $\left[ (-1,-1), \text{Left}\right]$ is $(-1,+1)$. The value ``--'' in the cell $[X,Y]$ indicates that a starving robot moving in direction $X$ does not hit the $Y$ wall, for example: $[(-1,-1), \text{Right}]=\text{--}$.

\begin{remark}\label{rem:changes_dirs_components}
The movement direction has two components, the first indicate the direction of the vertical movement and the second indicates the direction of the horizontal movement. The Table~\ref{tab:dirs_changes} shows that when a robot hits a horizontal wall (top or bottom) only changes its vertical direction, and when it hits a vertical wall (left or right) only changes its horizontal direction.
\end{remark}

\begin{theorem}\label{thm:starv_no_same_row}
Two starving robots can not be in circles that are in the same row in a grid communication graph.
\end{theorem}
\begin{proof}
Let $H$ be the simplified diagram of movement of a grid communication graph. Suppose that there exist two starving robots $a$ and $b$ in vertices with the same row in $H$. There are eight possible combinations of movement direction between them:
\begin{itemize}
\item if $a$ is moving in $(-1,-1)$ direction then $b$ is moving in $(-1,-1)$ or $(-1,+1)$ direction,
\item if $a$ is moving in $(-1,+1)$ direction then $b$ is moving in $(-1,+1)$ or $(-1,-1)$ direction,
\item if $a$ is moving in $(+1,-1)$ direction then $b$ is moving in $(+1,-1)$ or $(+1,+1)$ direction, and
\item if $a$ is moving in $(+1,+1)$ direction then $b$ is moving in $(+1,+1)$ or $(+1,-1)$ direction.
\end{itemize}
Note that the vertical direction of movement of $a$ and $b$ are equal. We focus in the columns of their respective positions, that is, the horizontal movement of them. Let $(i,j)$ and $(i,r)$ be the positions of $a$ and $b$ in $H$, respectively. Without loss of generality, assume that $j\leq r$. If $j=r$ then they are not starving, thus $j<r$ and $r-j=2k$ by Remark~\ref{rem:keeping_row}. Analyzing the possible cases we have that:

\begin{enumerate}
\item If the horizontal movement of $a$ is $+1$ and the horizontal movement of $b$ is $-1$, then, $a$ is moving toward the right wall and $b$ is moving toward the left wall. Note that $a$ is between the left wall and $b$, and $b$ is between $a$ and the right wall. Therefore, after $k$ quarters of a system period the column of $a$ is $j+k$ and the column of $b$ is $r-k$ because they have not changed their respective horizontal movement direction (they have not yet hit any vertical wall). At this instant, $a$ and $b$ are in the same vertex of $H$ because they always are in the same row and $j+k=r-k$, thus they are not starving.\label{case:1}

\item If the horizontal movement of $a$ is $+1$ and the horizontal movement of $b$ is $+1$, then, they are moving toward the right wall. Let $m$ be the index of the last column in $H$. During $m-r$ quarters of a system period, they keep the horizontal movement direction because they do not hit vertical walls. Thus, after $m-r$ a quarters of system period the column of $a$ is $j+m-r$ and the column of $b$ is $r+m-r=m$. At this instant, the horizontal distance between them is still $2k$, and the robot $b$ is hitting the right wall, thus it changes its horizontal movement direction to $-1$. Therefore from this instant they are in the previous case.

\item If the horizontal movement of $a$ is $-1$ and the horizontal movement of $b$ is $-1$, then, they are moving toward the left wall. This case is analogous to the previous one.

\item If the horizontal movement of $a$ is $-1$ and the horizontal movement of $b$ is $+1$, then, $a$ is moving toward the left wall and $b$ is moving toward the right wall. Let $m$ be the index of the last column in $H$, we have the following cases:
\begin{itemize}
\item[] If $j<m-r$, then after $j$ quarters of a system period the column of $a$ is 0 and the column of $b$ is $r+j$. At this instant, $a$ changes its horizontal movement direction to $+1$, $r+j<m$ then $b$ keeps its horizontal movement direction $+1$ and the distance between them is $2(k+j)$. Therefore, from this instant they are in the second case. 
\item[] If $j=m-r$, then after $j$ quarters of a system period the column of $a$ is 0 and the column of $b$ is $r+j=r+m-r=m$. At this instant, $a$ changes its horizontal movement direction to $+1$, $r+j=m$ then $b$ changes its horizontal movement direction to $-1$ and the distance between them is $2(k+j)$. Therefore, from this instant they are in the first case. 
\item[] If $j>m-r$, then after $m-r$ quarters of a system period the column of $a$ is $j-(m-r)$ and the column of $b$ is $r+(m-r)=m$. At this instant, $a$ keeps its horizontal movement direction  $-1$, $r+(m-r)=m$ then $b$ changes its horizontal movement direction to $-1$ and the distance between them is $2(k+m-r)$. Therefore, from this instant they are in the third case. 
\end{itemize}

\end{enumerate}
From the above analysis, the result follows.
\end{proof}

The following results are analogous to Lemma~\ref{lem:same_row_position}, Lemma~\ref{lem:same_row_always_same_row}, Remark~\ref{rem:keeping_row} and Theorem~\ref{thm:starv_no_same_row} for robots in circles that are in the same column in a grid communication graph. We skip the proofs because the arguments are similar.

\begin{lemma}\label{lem:same_column_position}
Let $a$ and $b$ be two robots in circles that are in the same column  in a grid communication graph. If the position of $a$ is $\alpha$, then the position of $b$ is $\alpha$ or $2\pi-\alpha$.
\end{lemma}

\begin{lemma}\label{lem:same_column_always_same_column}
Suppose that at time $t_0$ there exist two starving robots in circles that are in the same column. Then, at any time after $t_0$ they will be in circles that are in the same column.
\end{lemma}

\begin{remark}\label{rem:keeping_column}
Let $a$ and $b$ be two robots in circles that are in the same column. If the position of $a$ in the diagram of movement is $(i,j)$, then, the position of $b$ in the diagram of movement is $(r,j)$ with $|r-i|$ even.
\end{remark}

\begin{theorem}\label{thm:starv_no_same_column}
Two starving robots can not be in circles that are in the same column in a grid communication graph.
\end{theorem}

Then, we deduce the following result:
\begin{corollary}
In an $n\times m$ grid system  the starvation number is $\min(n,m)$.
\end{corollary}
\begin{proof}
Using theorems \ref{thm:starv_no_same_row} and \ref{thm:starv_no_same_column} we have that in a grid system there are at most a starving robot per column and a starving robot per row. Thus, by pigeonhole principle, $\min(n,m)$ is the maximum number of starving robots.
\end{proof}

From this and using Lemma~\ref{lem:resilience_rel_starv} we deduce the main result of this section:

\begin{theorem}\label{thm:grid_resilience}
In an $n\times m$ grid system the isolation-resilience is $n*m-\min(n,m)-1$.
\end{theorem}

\section{Conclusion and Future Work}\label{ch:conclusions}


In this work we have studied a combinatorial problem related with the robustness of synchronized system composed by robots that cooperate to cover an area with constrained communication range. Firstly, the theoretical results of  \cite{dbanez2015icra} about synchronized systems on a model formed by identical circular trajectories and robots flying at the same constant speed were presented. Then we stated the concept of \emph{starvation} as a phenomenon that can appear when a set of robots leaves a synchronized system. This phenomenon is characterized by the permanent loss of communication of one or more surviving agents when a set of robots leaves a synchronized system. Also, we present the \emph{starvation state} of a system as an extreme case of disconnection, where all the surviving robots in the system are permanently isolated. Then we introduced the main concept of this work, the \emph{resilience} of a system. More specifically, we defined the \emph{uncovering-resilience} and the \emph{isolation-resilience}.
They measure how resistant is a system to breaking the covering of all the trajectories and the starvation state. The resilience of the system is the maximum number of robots that can fail in the system without resulting these states.

We have proved that the resilience of a synchronized system does not depend on the time when the robots leave the system, or the direction of movement of the robots, or the starting positions of the robots; it depends only on the communication graph considered in the system. Then, we have focused on the study of systems in starvation state. The concept of \emph{ring} has been introduced to study the motion of starving robots in a system. We have presented some general results on rings that have then used to calculate the resilience (or at least some bounds) in specific cases.

By studying the communication graphs that are trees, we have proved that these graphs have only one ring and we found tight bounds for the isolation-resilience in these cases. Moreover, we have shown that if the communication graph is a simple chain formed by $n$ trajectories, then the resilience is $n-2$. 
%
The case cycle graphs has also been studied. We have shown that these graphs have only two rings and all the starving robots are in one of them. This property lead us to compute easily the isolation-resilience in these cases.

The \emph{grid communication graphs} are another interesting configuration that we have studied in this work. We have shown that the number of rings in these cases is exactly $\text{gcd}(n,m)$, where $n$ and $m$ are the number of rows and columns of the grid communication graph, respectively. We have proven that the isolation-resilience of these cases depends only on the dimensions of the grid communication graph and it is $n\cdot m-\min(n,m)-1$.

The main open problem is to design an efficient algorithm to compute the isolation-resilience of any communication graph. 


Another related possible research line is studying the starvation of exactly $k$ robots in the system.

\bibliography{bibliography}
\bibliographystyle{plain}
\end{document}